\documentclass[a4paper]{amsart}

\usepackage{cmbright}
\usepackage{mathrsfs}
\usepackage[utf8]{inputenc}
\usepackage[T1]{fontenc}
\usepackage[english]{babel}
\usepackage{amsfonts}
\usepackage{amsthm}
\usepackage{amssymb}
\usepackage{mathcomp}
\usepackage{mathrsfs}
\usepackage[bb=boondox]{mathalfa}
\usepackage{graphicx}
\usepackage{epstopdf}
\usepackage[centertags]{amsmath}
\usepackage[pdftex]{hyperref}
\usepackage{vmargin}
\usepackage{tikz}
\usepackage{tikz-cd}
\usepackage{dsfont}
\usepackage{cases}
\usepackage{mathtools}
\usepackage{shuffle}
\usepackage{adjustbox}

\newcommand{\mycomment}[1]{}

\hypersetup{
    colorlinks,
    linkcolor=[rgb]{0.733, 0.149, 0.286},
    citecolor=[rgb]{0.733, 0.149, 0.286},
    urlcolor=[rgb]{0.733, 0.149, 0.286}
}

\title{Symmetric operads from a planar point of view}
\author{Brice Le Grignou \quad Victor Roca i Lucio}

\address{Brice Le Grignou}
\email{\href{mailto:bricelegrignou@gmail.com}{bricelegrignou@gmail.com}}

\address{Victor Roca i Lucio, Ecole Polytechnique Fédérale de Lausanne, EPFL,
CH-1015 Lausanne, Switzerland}
\email{\href{mailto:victor.rocalucio@epfl.ch}{victor.rocalucio@epfl.ch}}

\date{\today}

\begin{document}

\theoremstyle{plain}
\newtheorem{theorem}{Theorem}
\newtheorem*{theorem*}{Theorem}
\newtheorem{lemma}{Lemma}
\newtheorem{proposition}{Proposition}
\newtheorem{assumption}{Assumption}
\newtheorem{corollary}{Corollary}

\theoremstyle{definition}
\newtheorem{definition}{Definition}
\newtheorem{hypothesis}{Hypothesis}

\theoremstyle{remark}
\newtheorem{remark}{\sc Remark}
\newtheorem{example}{\sc Example}
\newtheorem*{notation}{\sc Notation}


\newcommand{\qi}{\xrightarrow{ \,\smash{\raisebox{-0.65ex}{\ensuremath{\scriptstyle\sim}}}\,}}
\newcommand{\lqi}{\xleftarrow{ \,\smash{\raisebox{-0.65ex}{\ensuremath{\scriptstyle\sim}}}\,}}

\newcommand{\draftnote}[1]{\marginpar{\raggedright\textsf{\hspace{0pt} \tiny #1}}}
\newcommand{\ac}{{\scriptstyle \text{\rm !`}}}

\newcommand{\Ch}{\categ{Ch}}
\newcommand{\Catesmall}{\categ{Cat}_{\categ E, \mathrm{small}}}
\newcommand{\Operade}{\Operad_{\categ E}}
\newcommand{\Operadecprime}{\Operad_{\categ E}}
\newcommand{\Operadesmall}{\Operad_{\categ E, \mathrm{small}}}
\newcommand{\eII}{\mathcal{I}}
\newcommand{\ecateg}[1]{\mathcal{#1}}
\newcommand{\cmonlax}{\categ{CMon}_\lax}
\newcommand{\cmonoplax}{\categ{CMon}_\oplax}
\newcommand{\cmonstrong}{\categ{CMon}_\strong}
\newcommand{\cmonstrict}{\categ{CMon}_\strict}
\newcommand{\cat}{\mathrm{cat}}
\newcommand{\lax}{\mathrm{lax}}
\newcommand{\oplax}{\mathrm{oplax}}
\newcommand{\strong}{\mathrm{strong}}
\newcommand{\strict}{\mathrm{strict}}
\newcommand{\pl}{\mathrm{pl}}
\newcommand{\gr}{\mathrm{gr}}

\newcommand{\Cats}{\mathsf{Cats}}
\newcommand{\Functors}{\mathsf{Functors}}

\newcommand{\Ob}{\mathrm{Ob}}
\newcommand{\tr}{\mathrm{tr}}
\newcommand{\catch}{\mathsf{Ch}}
\newcommand{\categ}[1]{\mathsf{#1}}
\newcommand{\set}[1]{\mathrm{#1}}
\newcommand{\catoperad}[1]{\mathsf{#1}}
\newcommand{\operad}[1]{\mathcal{#1}}
\newcommand{\algebra}[1]{\mathrm{#1}}
\newcommand{\coalgebra}[1]{\mathrm{#1}}
\newcommand{\cooperad}[1]{\mathcal{#1}}
\newcommand{\ocooperad}[1]{\overline{\mathcal{#1}}}
\newcommand{\catofmod}[1]{{#1}\mathrm{-}\mathsf{mod}}
\newcommand{\catofcog}[1]{#1\mathrm{-}\mathsf{cog}}
\newcommand{\catcog}[1]{\Cog\left(#1\right)}

\newcommand{\catdgmod}[1]{\categ{dg}~#1\text{-}\categ{mod}}
\newcommand{\catpdgmod}[1]{\categ{pdg}~#1\text{-}\categ{mod}}
\newcommand{\catgrmod}[1]{\categ{gr}~#1\text{-}\categ{mod}}

\newcommand{\catdgalg}[1]{\categ{dg}~#1\text{-}\categ{alg}}
\newcommand{\catpdgalg}[1]{\categ{pdg}~#1\text{-}\categ{alg}}
\newcommand{\catgralg}[1]{\categ{gr}~#1\text{-}\categ{alg}}
\newcommand{\catcurvalg}[1]{\categ{curv}~#1\text{-}\categ{alg}}

\newcommand{\catdgcompalg}[1]{\categ{dg}~#1\text{-}\categ{alg}^{\mathsf{comp}}}
\newcommand{\catpdgcompalg}[1]{\categ{pdg}~#1\text{-}\categ{alg}^{\mathsf{comp}}}
\newcommand{\catgrcompalg}[1]{\categ{gr}~#1\text{-}\categ{alg}^{\mathsf{comp}}}
\newcommand{\catcurvcompalg}[1]{\categ{curv}~#1\text{-}\categ{alg}^{\mathsf{comp}}}

\newcommand{\catdgcog}[1]{\categ{dg}~#1\text{-}\categ{cog}}
\newcommand{\catpdgcog}[1]{\categ{pdg}~#1\text{-}\categ{cog}}
\newcommand{\catgrcog}[1]{\categ{gr}~#1\text{-}\categ{cog}}
\newcommand{\catcurvcog}[1]{\categ{curv}~#1\text{-}\categ{cog}}

\newcommand{\dgoperads}{\categ{dg}~\categ{Operads}}
\newcommand{\pdgoperads}{\categ{pdg}~\categ{Operads}}
\newcommand{\groperads}{\categ{gr}~\categ{Operads}}
\newcommand{\curvcooperads}{\categ{curv}~\categ{Cooperads}}
\newcommand{\grcooperads}{\categ{gr}~\categ{Cooperads}}

\newcommand{\pdgcooperads}{\categ{pdg}~\categ{Cooperads}}
\newcommand{\dgcooperads}{\categ{dg}~\categ{Cooperads}}

\newcommand{\catalg}[1]{\Alg\left(#1\right)}

\newcommand{\catofcolcomod}[1]{\mathsf{Col}\mathrm{-}{#1}\mathrm{-}\mathsf{comod}}
\newcommand{\catofcoalgebra}[1]{{#1}\mathrm{-}\mathsf{cog}}
\newcommand{\catofalg}[1]{\operad{#1}\mathrm{-}\mathsf{alg}}
\newcommand{\catofalgebra}[1]{{#1}\mathrm{-}\mathsf{alg}}
\newcommand{\mbs}{\mathsf{S}}
\newcommand{\catocol}[1]{\mathsf{O}_{\set{#1}}}
\newcommand{\catoftrees}{\mathsf{Trees}}
\newcommand{\cattcol}[1]{\catoftrees_{\set{#1}}}
\newcommand{\catcorcol}[1]{\mathsf{Corol}_{\set{#1}}}
\newcommand{\Einfty}{\mathcal{E}_{\infty}}
\newcommand{\nuEinfty}{\mathcal{nuE}_{\infty}}
\newcommand{\Fun}[3]{\mathrm{Fun}^{#1}\left(#2,#3\right)}
\newcommand{\III}{\operad{I}}
\newcommand{\treeoperad}{\mathbb{T}}
\newcommand{\treemodule}{\mathbb{T}}
\newcommand{\core}{\mathrm{Core}}
\newcommand{\forget}{\mathrm{U}}
\newcommand{\treemonad}{\mathbb{O}}
\newcommand{\cogcomonad}[1]{\mathbb{L}^{#1}}
\newcommand{\cofreecog}[1]{\mathrm{L}^{#1}}

\newcommand{\barfunctor}[1]{\mathrm{B}_{#1}}
\newcommand{\baradjoint}[1]{\mathrm{B}^\dag_{#1}}
\newcommand{\cobarfunctor}[1]{\mathrm{C}_{#1}}
\newcommand{\cobaradjoint}[1]{\mathrm{C}^\dag_{#1}}
\newcommand{\Operad}{\mathsf{Operad}}
\newcommand{\coOperad}{\mathsf{coOperad}}

\newcommand{\Aut}[1]{\mathrm{Aut}(#1)}

\newcommand{\verte}[1]{\mathrm{vert}(#1)}
\newcommand{\edge}[1]{\mathrm{edge}(#1)}
\newcommand{\leaves}[1]{\mathrm{leaves}(#1)}
\newcommand{\inner}[1]{\mathrm{inner}(#1)}
\newcommand{\inp}[1]{\mathrm{input}(#1)}

\newcommand{\field}{\mathbb{K}}
\newcommand{\mbk}{\mathbb{K}}
\newcommand{\mbn}{\mathbb{N}}

\newcommand{\id}{\mathrm{Id}}
\newcommand{\ii}{\mathrm{id}}
\newcommand{\unit}{\mathds{1}}

\newcommand{\Lin}{Lin}

\newcommand{\BijC}{\mathsf{Bij}_{C}}

\newcommand{\kk}{\Bbbk}
\newcommand{\PP}{\mathcal{P}}
\newcommand{\C}{\mathcal{C}}
\newcommand{\Sy}{\mathbb{S}}
\newcommand{\Tree}{\mathsf{Tree}}
\newcommand{\treemod}{\mathbb{T}}
\newcommand{\Dend}{\Omega}
\newcommand{\aDend}{\Omega^{\mathsf{act}}}
\newcommand{\cDend}{\Omega^{\mathsf{core}}}
\newcommand{\cDendpart}{\cDend_{\mathsf{part}}}

\newcommand{\build}{\mathrm{Build}}
\newcommand{\col}{\mathrm{col}}

\newcommand{\HOM}{\mathrm{HOM}}
\newcommand{\Hom}[3]{\mathrm{hom}_{#1}\left(#2 , #3 \right)}
\newcommand{\ov}{\overline}
\newcommand{\otimeshadamard}{\otimes_{\mathbb{H}}}

\newcommand{\I}{\mathcal{I}}
\newcommand{\Aa}{\mathcal{A}}
\newcommand{\BB}{\mathcal{B}}
\newcommand{\CC}{\mathcal{C}}
\newcommand{\DD}{\mathcal{D}}
\newcommand{\EE}{\mathcal{E}}
\newcommand{\FF}{\mathcal{F}}
\newcommand{\II}{\mathbb{1}}
\newcommand{\RR}{\mathcal{R}}
\newcommand{\UU}{\mathcal{U}}
\newcommand{\VV}{\mathcal{V}}
\newcommand{\WW}{\mathcal{W}}
\newcommand{\AAA}{\mathscr{A}}
\newcommand{\BBB}{\mathscr{B}}
\newcommand{\CCC}{\mathscr{C}}
\newcommand{\DDD}{\mathscr{D}}
\newcommand{\EEE}{\mathscr{E}}
\newcommand{\FFF}{\mathscr{F}}

\newcommand{\PPP}{\mathscr{P}}
\newcommand{\QQQ}{\mathscr{Q}}

\newcommand{\QQ}{\mathcal{Q}}

\newcommand{\KKK}{\mathscr{K}}
\newcommand{\KK}{\mathcal{K}}

\newcommand{\ra}{\rightarrow}

\newcommand{\Ai}{\mathcal{A}_{\infty}}
\newcommand{\uAi}{u\mathcal{A}_{\infty}}
\newcommand{\uEinfty}{u\mathcal{E}_{\infty}}
\newcommand{\uAW}{u\mathcal{AW}}

\newcommand{\uAlg}{\mathsf{Alg}}
\newcommand{\nuAlg}{\mathsf{nuAlg}}
\newcommand{\cAlg}{\mathsf{cAlg}}

\newcommand{\ucAlg}{\mathsf{ucAlg}}
\newcommand{\Cog}{\mathsf{Cog}}
\newcommand{\nuCog}{\mathsf{nuCog}}
\newcommand{\uAWcog}{u\mathcal{AW}-\mathsf{cog}}

\newcommand{\uCog}{\mathsf{uCog}}
\newcommand{\cCog}{\mathsf{cCog}}
\newcommand{\ucCog}{\mathsf{ucCog}}
\newcommand{\cNilCog}{\mathsf{cNilCog}}
\newcommand{\ucNilCog}{\mathsf{ucNilCog}}
\newcommand{\NilCog}{\mathsf{NilCog}}

\newcommand{\Cocom}{\mathsf{Cocom}}
\newcommand{\uCocom}{\mathsf{uCocom}}
\newcommand{\NilCocom}{\mathsf{NilCocom}}
\newcommand{\uNilCocom}{\mathsf{uNilCocom}}
\newcommand{\Liealg}{\mathsf{Lie}-\mathsf{alg}}
\newcommand{\cLiealg}{\mathsf{cLie}-\mathsf{alg}}
\newcommand{\Alg}{\mathsf{Alg}}
\newcommand{\Linfty}{\mathcal{L}_{\infty}}
\newcommand{\CMC}{\mathfrak{CMC}}
\newcommand{\Tfree}{\mathbb{T}}

\newcommand{\Hinich}{\mathsf{Hinich} -\mathsf{cog}}

\newcommand{\Ccomod}{\mathscr C -\mathsf{comod}}
\newcommand{\Pmod}{\mathscr P -\mathsf{mod}}

\newcommand{\cCoop}{\mathsf{cCoop}}

\newcommand{\Set}{\mathsf{Set}}
\newcommand{\sSet}{\mathsf{sSet}}
\newcommand{\dgMod}{\mathsf{dgMod}}
\newcommand{\gMod}{\mathsf{gMod}}
\newcommand{\catOrd}{\mathsf{Ord}}
\newcommand{\catBij}{\mathsf{Bij}}
\newcommand{\catSmod}{\mbs\mathsf{mod}}
\newcommand{\EEtw}{\mathcal{E}\text{-}\mathsf{Tw}}
\newcommand{\OpBim}{\mathsf{Op}\text{-}\mathsf{Bim}}

\newcommand{\Palg}{\mathcal{P}-\mathsf{alg}}
\newcommand{\Qalg}{\mathcal{Q}-\mathsf{alg}}
\newcommand{\Pcog}{\mathcal{P}-\mathsf{cog}}
\newcommand{\Qcog}{\mathcal{Q}-\mathsf{cog}}
\newcommand{\Ccog}{\mathcal{C}-\mathsf{cog}}
\newcommand{\Dcog}{\mathcal{D}-\mathsf{cog}}
\newcommand{\uCoCog}{\mathsf{uCoCog}}

\newcommand{\Artinalg}{\mathsf{Artin}-\mathsf{alg}}

\newcommand{\colim}[1]{\underset{#1}{\mathrm{colim}}}
\newcommand{\Map}{\mathrm{Map}}
\newcommand{\Def}{\mathrm{Def}}
\newcommand{\Bij}{\mathrm{Bij}}
\newcommand{\op}{\mathrm{op}}

\newcommand{\undern}{\underline{n}}
\newcommand{\dginterval}{{N{[1]}}}
\newcommand{\dgsimplex}[1]{{N{[#1]}}}

\newcommand{\cofree}{ T^c}
\newcommand{\Tw}{ Tw}
\newcommand{\End}{\mathcal{E}\mathrm{nd}}
\newcommand{\catEnd}{\mathsf{End}}
\newcommand{\coEnd}{\mathrm{co}\End}
\newcommand{\Mult}{\mathrm{Mult}}
\newcommand{\coMult}{\mathrm{coMult}}

\newcommand{\Lie}{\mathcal{L}\mathr{ie}}
\newcommand{\As}{\mathcal{A}\mathrm{s}}
\newcommand{\uAs}{\mathrm{u}\As}
\newcommand{\coAs}{\mathrm{co}\As}
\newcommand{\Com}{\mathcal{C}\mathrm{om}}
\newcommand{\uCom}{\mathrm{u}\Com}
\newcommand{\Perm}{\catoperad{Perm}}
\newcommand{\uBE}{\mathrm{u}\mathcal{BE}}
\newcommand{\uBEs}{{\uBE}^{\mathrm s}}

\newcommand{\comp}{\circ}
\newcommand{\restrictionextension}{\mathrm{RE}}
\newcommand{\extension}{\mathrm{E}}
\newcommand{\extensionone}{\mathrm{E}_1}
\newcommand{\extensiontwo}{\mathrm{E}_2}
\newcommand{\restrictionone}{\mathrm{R}_1}
\newcommand{\restrictiontwo}{\mathrm{R}_2}

\newcommand{\itemt}{\item[$\triangleright$]}

\newcommand{\poubelle}[1]{}

\newcommand{\Victor}[1]{\textcolor{blue}{#1}}
\newcommand{\Brice}[1]{\textcolor{red}{#1}}

\maketitle

\begin{abstract}
The aim of this note is to give a detailed account of how symmetric operads can be constructed from planar (non-symmetric) operads, and to carefully spell out the algebraic interplay between these two notions. It is a companion note to the main paper \cite{premierpapier}. 
\end{abstract}

\setcounter{tocdepth}{1}
\tableofcontents
\setcounter{tocdepth}{1}
\tableofcontents

\section*{Introduction}
This is a companion note to the main paper \cite{premierpapier}. In \textit{op.cit.}, we generalize the homotopical methods of operadic calculus over a field of positive characteristic. The goal of this note is to provide the necessary algebraic background needed in order to understand how symmetric operads and cooperads relate to their planar counterparts. Let us explain why this becomes a key issue over a positive characteristic field.

\medskip

Symmetric (co)operads carry an action of the symmetric group $\mathbb{S}_n$ for all $n \geq 0$, whereas their planar (also called non-symmetric) counterpart do not. The representation theory of $\mathbb{S}_n$ is fairly well-behaved when $\kk$ is a characteristic zero field, their category almost behaves like the category of $\kk$-modules. This makes the theory of symmetric (co)operads behave in a very similar fashion to their planar analogues.

\medskip

Nevertheless, things change drastically over a positive characteristic field. Understanding the representation theory of the symmetric groups in this setting is an active subject of research, see for instance \cite{ICMSymmetric}. Hence, over a field of positive characteristic, the theory of symmetric (co)operads becomes much harder than the theory of planar (co)operads. However, almost all of the algebraic structures of interest can only be encoded with symmetric operads, since operations usually have symmetries. 

\medskip

Fortunately, these two notions are not unrelated: they are in fact connected to each other by a quite rich algebraic interplay. Our aim in this note is to layout the algebraic structures that relate these two notions. We begin, in the first section, by recalling the main constructions in the planar setting: the composition product, the tree monad and the reduced tree comonad. We refer to \cite{LodayVallette} for more details. Then we construct, in the second section, the composition product of $\mathbb{S}$-modules only using its planar analogue and the free-forgetful adjunction between $\mathbb{S}$-modules and $\mathbb{N}$-modules; and we show that our construction coincides with the classical composition product of $\mathbb{S}$-modules. Therefore our definitions for symmetric (co)operads coincide with the ones used in the literature. In the third section, we construct the tree endofunctor from the planar tree endofunctor and show it has a monad structure, deeply related to the planar monad structure. This is crucial in order to understand symmetric dg operads which are free as symmetric operads and whose generators are free $\mathbb{S}$-modules. 

\medskip

Finally, the most complicated part is dealing with the reduced tree endofunctor and its comonad structure in section four. We construct its comonad structure from its planar analogue and relate the two constructions. Since \textit{conilpotent} cooperads are precisely coalgebras over the reduced tree comonad, this allows us to understand symmetric conilpotent dg cooperads whose underlying conilpotent cooperad is cofree generated by a free $\mathbb{S}$-module in planar terms. Understanding these objects is crucial in \cite{premierpapier}, as the main definition of \textit{op.cit.}, the one that makes the theory work, is that of a \textit{quasi-planar conilpotent dg cooperad}. This definition can be understood as a particular example of a symmetric conilpotent dg cooperads whose underlying conilpotent graded cooperad is particularly well-behaved with respect to the symmetric group actions. And many examples of such cooperads are provided by symmetric conilpotent dg cooperads whose underlying conilpotent graded cooperad is cofree generated by a free $\mathbb{S}$-module, but where the differential (which might interact non-trivially with the symmetric groups) is still particularly well-behaved with respect to these actions. We will actually relate these quasi-planar cooperads to \textit{higher cooperads} described in \cite{BrunoMalte}. Let us end this introduction by point out that, although the case over the integers $\mathbb{Z}$ or the $p$-adic numbers $\mathbb{Z}_p$ is more difficult that the case over a field $\kk$, the point of view adopted in this note could be useful in order to further generalize the theory of algebraic operads and its main methods to these settings.

\subsection*{Conventions}
We work in the category symmetric monoidal category of differential graded modules (chain complexes) over a field $\kk$. All the results of this note are also true in the graded or in pre-differential graded setting. See \cite{premierpapier} for more details on the setting. We will use $X,Y$ for generic dg $\mathbb{N}$-modules and $M,N$ for generic dg $\mathbb{S}$-modules. 

\section{The planar case}
We give some recollections about the planar operads, which can also be found in \cite[Chapter 5, Section 9]{LodayVallette}. We then consider conilpotent planar cooperads, where we define them as coalgebras over the reduced planar tree comonad. 

\subsection{$\mathbb N$-modules, planar operads and planar cooperads}
We consider $\mathbb N$ as a category where objects are natural integers and where there are only identity morphisms. 

\begin{definition}[dg $\mathbb N$-modules]
    A dg $\mathbb N$-modules amounts to the data of a functor 
    \[
    X: \mathbb N \longrightarrow \catdgmod{\kk}.
    \]    
    The object $X(n)$ is called the arity $n$ part of $X$. We denote $\catdgmod{\mathbb N}$
    the category of dg $\mathbb N$-modules. 
\end{definition}

The planar horizontal product on dg $\mathbb N$-modules $X,Y$ is given by the Day convolution

\[
(X \circledast_\pl Y)(n) \coloneqq \bigoplus_{k+l = n} X(k) \otimes Y(l)~.
\]
\vspace{0.1pc}

This endows dg $\mathbb N$ with a symmetric monoidal category structure, where the unit is given by $\kk$ concentrated in arity $0$.

\medskip

There is another monoidal structure given by the planar composition product 

\[
(X \comp_\pl Y)(n) \coloneqq \bigoplus_{k \geq 0} X(k) \otimes Y^{\circledast_\pl k} (n)~.
\]
\vspace{0.1pc}

The unit for the composition is given by $\operad I$, defined as follows 
$$
\II (n) \coloneqq
\begin{cases}
    0 \text{ if }n \neq 1~,
    \\
    \kk \text{ if }n = 1.
\end{cases}
$$

\begin{definition}[Planar dg operad]
A \textit{planar dg operad} $\operad {P}$ amounts to the data of a monoid $(\mathcal{P},\gamma,\eta)$ in the category of dg $\mathbb N$-modules with respect to the composition product. 
\end{definition}

\begin{definition}[Augmented planar dg operad]
An augmented planar dg operad $\operad P$ amounts to the data of a planar dg operad $(\mathcal{P},\gamma,\eta)$ equipped with a morphism of planar dg operads $\nu: \operad P \longrightarrow \operad I$ such that $\nu \circ \eta = \mathrm{id}.$
\end{definition}

Given an augmented planar dg operad $\operad P$, we will denote by $\overline{\operad P}$ the kernel of the augmentation map. 

\begin{definition}[Planar dg cooperad]
A \textit{planar dg cooperad} $\operad C$ amounts to the data of a comonoid $(\C, \Delta, \epsilon)$ in the category of dg $\mathbb N$-modules with respect to the composition product. 
\end{definition}

Given a planar dg cooperad $\C$, we will denote by $\overline{\operad C}$ the kernel of the counit map. 

\begin{definition}[Coaugmented planar dg cooperad]
A coaugmented planar dg cooperad $\operad C$ amounts to the data of a planar dg cooperad $(\C, \Delta, \epsilon)$ equipped with a morphism of planar dg cooperads $\mu: \operad I \longrightarrow \operad C$ such that $\epsilon \circ \mu = \mathrm{id}$. 
\end{definition}


\subsection{The planar tree module}
In the subsequent subsections, we will extensively use the notion of a tree, which we define in Appendix \ref{appendixtrees}. Let $t$ be a planar tree with $n$ leaves. It induces an endofunctor $t(-)$ in the category of dg $\mathbb N$-modules. For $m \geq 0$, it is given by 

\[
t(X)(m) =
\begin{cases}
 X(n_1) \otimes \cdots \otimes X(n_k) \quad \text{when m = n}~,
 \\
 0 \quad \text{otherwise}~,
\end{cases}
\]
\vspace{0.1pc}

where $t$ has $k$ nodes and its nodes have, in order, $n_1, \ldots, n_k$ outputs. 

\medskip

For every dg $\mathbb N$-module $X$, one can define the \textit{planar tree module} $\treemod_\pl(X)$ of $X$, which is the dg $\mathbb N$-module given, for $m \geq 0$, by
\[
\treemod_\pl(X)(m)  = \bigoplus_t t(X)~,
\]
where the sum is taken over the isomorphism classes of planar trees with $m$ leaves.

\medskip

Let $n$ be in $\mathbb N$. We also define the following subfunctors of the planar tree module:

\medskip

\begin{enumerate}
	\item The \textit{reduced planar tree endofunctor} $\overline{\treemod}_\pl(X)$, given by the sum over all non-trivial planar trees;
	
\medskip

    \item The $n$\textit{-levelled planar tree endofunctor} $\treemod_{\pl, \leq n}(X)$, given by the sum over planar trees whose height is equal or lower than $n$ (recall that the height of the trivial tree with no node is $0$);
    
\medskip

    \item The $n$-\textit{weight planar tree endofunctor} $\treemod_{\pl}^{(\leq n)}(X)$, given by the sum over planar trees with $n$ nodes or less.
 
\end{enumerate}

\medskip

All these constructions are natural in $X$ and define endofunctors of the category of dg $\mathbb N$-modules.

\begin{remark}
One can combine these notations. For instance the $n$-levelled reduced planar tree endofunctor $\overline{\treemod}_{\pl,\leq n}(X)$ is obtained by taking the sum over non trivial trees whose height is at most $n$. 
\end{remark}

\begin{remark}
Notice that, for all $n \geq 0$, there is a canonical isomorphism
\[
\treemod_{\pl, \leq n+1} X \cong \operad I \oplus X \comp_\pl \treemod_{\pl,\leq n} X.
\]
that yields an isomorphism $\treemod_\pl X \cong \operad I \oplus X \comp \treemod_{\pl} X$.
\end{remark}

The following two propositions follows from a straightforward checking.

\begin{proposition}
The planar tree module endofunctor $\treemod_\pl$ admits the following monad structure:

\medskip

\begin{enumerate}
\item The unit $\eta_X: X \longrightarrow \treemod_\pl X$ is the canonical inclusion, where elements in $X(m)$ are seen as $m$-corollas.

\medskip

\item The product $\gamma_X: \treemod_\pl~\treemod_\pl(X) \longrightarrow \treemod_\pl (X)$ is given, for a planar tree $t$ with $n$ leaves, as follows
\[
\begin{tikzcd}[column sep=0.5pc,row sep=2.5pc]
    t(\treemod_\pl X)(n) \coloneqq (\treemod_\pl X)(n_1) \otimes \cdots \otimes (\treemod_\pl X)(n_k)
    \ar[d,"\cong"] \\
	\displaystyle \bigoplus_{t_1, \ldots, t_k} t_1(X)(n_1) \otimes \cdots \otimes t_k(X)(n_k)
	\ar[d] \\
	\displaystyle \bigoplus_{t_1, \ldots, t_k} \mathrm{graft}_t(t_1, \ldots, t_k)(X)(n)~,
\end{tikzcd}
\]
where $\mathrm{graft}_t(t_1, \ldots, t_k)$ is the planar tree obtained from $t$ by replacing each node $i$ by the planar tree $t_i$ which labelled it.
\end{enumerate}
\end{proposition}

\begin{proposition}
    The category of algebras over the monad $\treemod_\pl$ is canonically isomorphic to the category of planar dg operads.
\end{proposition}

\begin{proof}
The restriction of the structural map to $\treemod_{\pl, \leq 2}$ gives the composition product, and induces an operad structure. Any operad structure can be extended to a full algebra structure over $\treemod_\pl$ by induction on the level of planar trees.
\end{proof}


\subsection{Planar conilpotent cooperads}
Conilpotent dg cooperads correspond exactly to those that arise as coalgebras over the reduced planar tree comonad. 

\begin{proposition}
The reduced planar tree module endofunctor $\overline{\treemod}_\pl$ admits the following comonad structure:

\medskip

\begin{enumerate}
\item The counit map $\epsilon_X: \overline{\treemod}_\pl(X) \longrightarrow X$ is the projection onto corollas labelled by $X$. 

\medskip

\item The coproduct $\Delta_X: \overline{\treemod}_\pl(X) \longrightarrow  \overline{\treemod}_\pl ~\overline{\treemod}_\pl(X)$ is given by partitioning non-trivial planar trees into non-trivial planar sub-trees. For a planar tree $t$ with $n$ leaves, the coproduct is given by the finite sum over partitions $t_1, \ldots, t_m$ of $t$ into non trivial sub-trees of the maps
\[
\begin{tikzcd}[column sep=0.5pc,row sep=2.5pc]
    t(X)(n)
    \ar[d, "\cong"]
    \\ t_1(X)(n_1) \otimes \cdots \otimes t_m(X)(n_m)
    \ar[d]
    \\ \overline{\treemod}_\pl(X)(n_1) \otimes \cdots \otimes \overline{\treemod}_\pl(X)(n_m)
    \ar[d, "\cong"]
    \\  t/(t_1, \ldots, t_m)(\overline{\treemod}_\pl(X))
    \ar[d, hookrightarrow]
    \\\overline{\treemod}_\pl\overline{\treemod}_\pl(X)~.
\end{tikzcd}
\]
\end{enumerate}
\end{proposition}

\begin{proof}
This follows from a straightforward checking.    
\end{proof}

\begin{corollary}\label{corollarydecompositionplanarccooperad}
For every natural integer $n \geq 1$, the comonad structure on $\overline{\treemod}_\pl$ restricts to $\overline{\treemod}_{\pl,\leq n}$ and $\overline{\treemod}_\pl^{(\leq n)}$, making the natural inclusions 
\[
\overline{\treemod}_{\pl,\leq n} \rightarrowtail \overline{\treemod}_\pl~, \quad \quad \overline{\treemod}_\pl^{(\leq n)} \rightarrowtail \overline{\treemod}_\pl~,
\]
morphisms of comonads.
\end{corollary}

\begin{proof}
This follows from the definition of the coproduct of the comonad. For every non-trivial tree $t$ with at most $n$-nodes (resp. of height at most $n$) and every partition $t_1, \ldots, t_m$ of $t$, all the planar trees $t_1, \ldots, t_m, t/(t_1, \ldots, t_m)$ have at most $n$-nodes (resp. have height at most $n$).
\end{proof}

A $\overline{\treemod}_\pl$-coalgebra structure induces a comonoid structure for the planar composition product of dg $\mathbb N$-modules.

\begin{lemma}\label{lemma: T-cog implique comonoide pour le circ}
For every dg $\overline{\treemod}_\pl$-coalgebra $(W,\delta_W)$, the following square
    \[
    \begin{tikzcd}[column sep=2.5pc,row sep=2.5pc]
        W
        \ar[r,"\delta_W"] \ar[d,"\delta_W",swap]
        & \overline{\treemod}_{\pl} W
        \ar[r, two heads]
        & \overline{\treemod}_{\pl,\leq 2} W
        \ar[r, "\cong"]
        & W \comp (\operad I \oplus W)
        \ar[d,rightarrowtail]
        \\
        \overline{\treemod}_{\pl} W
        &&&  W \comp {\treemod}_{\pl} W.
        \ar[lll, "\cong",swap]
    \end{tikzcd}
    \]
is a commutative diagram.
\end{lemma}

\begin{proof}
    This square decomposes into the diagram
    $$
    \begin{tikzcd}[column sep=2pc,row sep=2pc]
        W
        \ar[r, "\delta_W"] \ar[d, "\delta_W"']
        & \overline{\treemod}_{\pl} W
        \ar[r, two heads]\ar[d, "\overline{\treemod}_\pl(\delta_W)"]
        & \overline{\treemod}_{\pl,\leq 2} W
        \ar[r, "\cong"] \ar[d, "\overline{\treemod}_{\pl, \leq 2}(\delta_W)"]
        & W \comp_\pl (W \oplus \operad I)
        \ar[d, "\delta_W \comp_\pl \delta_W"]
        \\
        \overline{\treemod}_{\pl} W
        \ar[r, "\delta_W"]
        &\overline{\treemod}_{\pl} \overline{\treemod}_{\pl} W
        \ar[r, two heads]
        & \overline{\treemod}_{\pl, \leq 2} \overline{\treemod}_{\pl} W
        \ar[r, "\cong"]
        & \overline{\treemod}_{\pl} W \comp_\pl  (\operad I \oplus \overline{\treemod}_{\pl})
        \ar[d, "\cong"]
        \\ 
        &&& \overline{\treemod}_{\pl} W \comp_\pl  {\treemod}_{\pl} W
        \ar[d, two heads]
        \\ 
        &&&  W \comp_\pl  {\treemod}_{\pl} W~,
        \ar[uulll, "\cong"]
    \end{tikzcd}
    $$
    of which all the cells are commutative.
\end{proof}

Let $(W,\delta_W)$ be a dg $\overline{\treemod}_\pl$-coalgebra. We consider the map
\[
\Delta_W: \operad I \oplus W \longrightarrow (\operad I \oplus W)\comp_\pl (\operad I \oplus W)
\]
defined as the sum of the maps

\medskip

\begin{tikzcd}[column sep=1.5pc,row sep=0.5pc]
    \operad I \oplus W \arrow[r,twoheadrightarrow]
    &W \arrow[r,"\delta_W"]
    &\overline{\treemod}_\pl W \arrow[r,twoheadrightarrow]
    &\overline{\treemod}_{\pl, \leq 2} W \cong W \comp_\pl (\operad I \oplus W) \arrow[r,rightarrowtail]
    &(\operad I \oplus W)\comp_\pl (\operad I \oplus W)~,
\end{tikzcd}

\begin{tikzcd}[column sep=2pc,row sep=0.5pc]
    \operad I \oplus W \arrow[r,twoheadrightarrow]
    &W \cong \operad I \comp_\pl W \arrow[r,rightarrowtail]
    &(\operad I \oplus W)\comp_\pl (\operad I \oplus W)~,
\end{tikzcd}

\begin{tikzcd}[column sep=2pc,row sep=0.5pc]
    \operad I \oplus W \arrow[r,twoheadrightarrow]
    &\operad I \arrow[r]
    &\operad I \comp_\pl \operad I~.
\end{tikzcd}

\medskip

Together with the counit $\epsilon_W: \operad I \oplus W \twoheadrightarrow \operad I$ and the coaugmentation $\mu_W: \operad I \rightarrowtail \operad I \oplus W$, they form a coaugmented planar cooperad structure on the dg $\mathbb N$-module $\operad I \oplus W$.

\medskip

This defines a functor 
\[
\mathrm{Conil}: \mathsf{dg}~\overline{\treemod}_\pl\text{-}\mathsf{cog} \longrightarrow (\dgcooperads_\pl)_{\operad I/}
\]
from dg $\overline{\treemod}_\pl$-coalgebras to coaugmented planar dg cooperads.

\begin{proposition}\label{propositionplanarccooperadff}
    The functor $\mathrm{Conil}$ from dg $\overline{\treemod}_\pl$-coalgebras to coaugmented planar dg cooperads
    is fully faithful.
\end{proposition}

\begin{proof}
This functor is faithful since its composition with the forgetful functor towards dg $\mathbb N$-modules is faithful.
    Let us prove that it is is full.
    Let $(V,\delta_V)$ and $(W,\delta_W)$ be two dg $\overline{\treemod}_\pl$-coalgebras
    and let us consider a morphism of coaugmented planar dg cooperads
    $f: \operad I \oplus V \to \operad I \oplus W$. Proving that
    $f$ stems from a morphism of dg $\overline{\treemod}_\pl$-coalgebras amounts to prove that the restricted map $\overline{f}: V \to W$ is a morphism of dg $\overline{\treemod}_\pl$-coalgebras.
    
    \medskip
    
    Since the natural map $\overline{\treemod}_\pl X
    \to \lim_{n \in \omega^\op}\overline{\treemod}_{\pl,\leq n} X$ is a monomorphism, it suffices to prove that for every natural integer $n \geq 1$
    the diagram
    $$
    \begin{tikzcd}[column sep=2.5pc,row sep=2.5pc]
        V
        \ar[r, "\overline{f}"] \ar[d,"\delta_V",swap]
        & W \ar[d,"\delta_W"]
        \\
        \overline{\treemod}_\pl V
        \ar[d,two heads]
        & \overline{\treemod}_\pl W
        \ar[d, two heads]
        \\
        \overline{\treemod}_{\pl,\leq n} V
        \ar[r, "\overline{\treemod}_{\pl,\leq n}(\overline{f})"']
        & \overline{\treemod}_{\pl,\leq n} W
    \end{tikzcd}
    $$
    is commutative. 
    
    \medskip
    
    For $n=1$, this square is clearly commutative. Let us assume that it is commutative for an integer $n \geq 1$ and let us prove that is is commutative for $n+1$. The square above for $n+1$ rewrites as the external square of the diagram
    $$
    \begin{tikzcd}[column sep=2.5pc,row sep=2.5pc]
        V
        \ar[r] \ar[d]
        & W
        \ar[d]
        \\
        V \comp_\pl (I \oplus V)
        \ar[r] \ar[d]
        & W \comp_\pl (I \oplus W)
        \ar[d]
        \\
        V \comp_\pl {\treemod}_{\pl,\leq n} V
        \ar[r] \ar[d, "\cong"']
        & W \comp_\pl {\treemod}_{\pl,\leq n} W
        \ar[d, "\cong"]
        \\
        \overline{\treemod}_{\pl,\leq n+1} V
        \ar[r]
        & \overline{\treemod}_{\pl,\leq n+1} W~,
    \end{tikzcd}
    $$
    of which all cells are commutative. So we have proven by induction that such a square is commutative for every natural integer $n \geq 1$. Therefore $\overline{f}$ is a morphism of dg $\overline{\treemod}_\pl$-coalgebras.
\end{proof}

\begin{definition}[Conilpotent planar dg cooperad]
Let $\C$ be a coaugmented planar dg cooperad. It is \textit{conilpotent} if it belongs to the essential image of the functor $\mathrm{Conil}$ from dg $\overline{\treemod}_\pl$-coalgebras to planar dg cooperads. We denote $\dgcooperads^{\categ{conil}}_\pl$ the full sub-category of coaugmented planar dg cooperads spanned by conilpotent ones.
\end{definition}

\begin{remark}
The idea behind this definition is the following: a cooperad can also be described in terms of partial decomposition maps
\[
\Delta_i: \C(n+k-1) \longrightarrow \C(n) \otimes \C(k)~,
\]
and it is conilpotent if and only if any iteration of these partial decompositions is eventually trivial. If this is the case, then the data of all the possible iterations is exactly encoded by the $\overline{\treemod}_\pl$-coalgebra structure. 
\end{remark}

\begin{proposition}\label{proppresentabilityofplanarconilp}
    Let $n$ be a natural integer. The forgetful functor 
    \[
    \mathsf{dg}~\overline{\treemod}_{\pl}^{(\leq n)}\text{-}\mathsf{cog} \longrightarrow \mathsf{dg}~\overline{\treemod}_{\pl}\text{-}\mathsf{cog}~.
    \]
    is fully faithful. Furthermore, both categories are presentable, and therefore it admits a right adjoint denoted by $\overline{\mathrm{F}}_n$.
\end{proposition}

\begin{proof}
The two categories are presentable. It follows from the fact that they are categories of coalgebras over accessible comonads on a presentable category. We refer to \cite{ChingRiehl} for more details. The fact that the functor is fully faithful follows from the fact that the morphism of comonads $\overline{\treemod}_{\pl}^{(\leq n)} \longrightarrow \overline{\treemod}_{\pl}$ is object-wise a monomorphism. The fact that it has a right adjoint is a direct consequence of the adjoint lifting theorem, see \cite[Appendix A]{premierpapier}.
\end{proof}

We will denote by $\mathrm{F}^{\mathrm{rad}}_n$ the following endofunctor 
\[
\begin{tikzcd}
\dgcooperads^{\categ{conil}}_\pl \arrow[r,"\cong"]
&\mathsf{dg}~\overline{\treemod}_{\pl}\text{-}\mathsf{cog} \arrow[r,"\overline{\mathrm{F}}_n"]
&\mathsf{dg}~\overline{\treemod}_{\pl}^{(\leq n)}\text{-}\mathsf{cog}  \arrow[r]
&\mathsf{dg}~\overline{\treemod}_{\pl}\text{-}\mathsf{cog}  \arrow[r,"\cong"]
&\dgcooperads^{\categ{conil}}_\pl 
\end{tikzcd}
\]

of the category of conilpotent planar dg cooperads.

\begin{definition}[Coradical filtration]
Let $\C$ be a conilpotent planar dg cooperad. Its $n$-\textit{coradical filtration} is given by the conilpotent planar dg cooperad $\mathrm{F}^{\mathrm{rad}}_n \C$. It induces a ladder diagram 
\[
\mathrm{F}^{\mathrm{rad}}_0 \C \rightarrowtail \mathrm{F}^{\mathrm{rad}}_1 \C \rightarrowtail \cdots \mathrm{F}^{\mathrm{rad}}_n \C \rightarrowtail \cdots~,
\]
indexed by $\mathbb{N}$, where all the arrows are monomorphisms.
\end{definition}

For any  conilpotent planar dg cooperad $\C$, there is a canonical isomorphism between $\C$ and the colimit of the following ladder diagram
\[
\mathrm{F}^{\mathrm{rad}}_0 \C \rightarrowtail \mathrm{F}^{\mathrm{rad}}_1 \C \rightarrowtail \cdots \mathrm{F}^{\mathrm{rad}}_n \C \rightarrowtail \cdots
\]
in the category of conilpotent planar dg cooperads.

\begin{proposition}
Let $n \geq 0$ and $\C$ be a conilpotent planar dg cooperad. Then $\mathrm{F}^{\mathrm{rad}}_n \C$ fits in the following pullback 
\[
        \begin{tikzcd}[column sep=2.5pc,row sep=2.5pc]
           \mathrm{F}^{\mathrm{rad}}_n \C \arrow[dr, phantom, "\lrcorner", very near start]
            \ar[r] \ar[d]
            & \overline{\treemod}_{\pl}^{(\leq n)} \C
            \ar[d,rightarrowtail]
            \\
            \C
            \ar[r,"\delta_\C"]
            & \overline{\treemod}_{\pl} \C.
        \end{tikzcd}
\]
in the category of dg $\mathbb{N}$-modules, where $\delta_\C$ denotes the dg $\overline{\treemod}_{\pl}$-coalgebra structure of $\C$.
\end{proposition}

\begin{proof}
This directly follows from the results in \cite[Appendix A]{premierpapier}.
\end{proof}


\section{Symmetric operads and cooperads} 
In this section, we construct the composition product of $\mathbb S$-modules, define operads, define cooperads, define the tree monad, define the reduced tree comonad, from their analogue planar notions. We do so using the free-forget adjunction between $\mathbb{N}$-modules and $\mathbb S$-modules and its algebraic properties.

\subsection{$\mathbb S$-modules}
\label{sectioncomositionproduct}
In this subsection, we deal with dg $\mathbb S$-modules. These correspond to collections of dg modules $\{M(n)\}$ for $n \geq 0$, where each $M(n)$ is endowed with an action of $\mathbb{S}_n$. The idea is to study them through the lens of dg $\mathbb{N}$-modules, using the adjunction that relates the two categories.

\begin{definition}[dg $\mathbb S$-module]
    Let $\mathbb S$ be the groupoid whose objects are natural integers and whose morphisms are given by 
    $$
    \hom_{\mathbb S}(n,m) = 
    \begin{cases}
        \emptyset \text{ if }n \neq m~,
        \\
        \mathbb S_n \text{ if }n = m~.
    \end{cases}
    $$
    A \textit{dg} $\mathbb S$\textit{-module} $M$ amounts to the data of a functor 
   	\[
   	M: \mathbb S^\op \longrightarrow \catdgmod{\kk}
   	\]
    from $\mathbb S^\op$ to dg modules. We denote by $\mathsf{dg}~\mathbb{S}\text{-}\mathsf{mod}$ the category of dg $\mathbb{S}$-modules. 
\end{definition}

There is a diagram of adjunctions 
\[
\begin{tikzcd}[column sep=6pc,row sep=2pc]
\mathsf{dg}~\mathbb{S}\text{-}\mathsf{mod}
\arrow[r,"U"{name=B, below}]
&\mathsf{dg}~\mathbb{N}\text{-}\mathsf{mod}~,
\arrow[l,bend right=30,"-\otimes~\mathbb{S}",swap,""{name=A,above}] \arrow[l,bend left=30,"-\otimes~\mathbb{S}",""{name=C,above}] \arrow[phantom, from=B, to=A, "\dashv" rotate=-90] \arrow[phantom, from=C, to=B, "\dashv" rotate=-90]
\end{tikzcd}
\]

between the categories of dg $\mathbb{N}$-modules and of dg $\mathbb{S}$-modules. The functor $- \otimes \mathbb S$ is given by 
\[
(X \otimes \mathbb S)(n) \coloneqq X(n) \otimes \kk[\mathbb S_n]~,
\]
for all $n \geq 0$. We will also denote by $- \otimes \mathbb S$ the endofunctor of dg $\mathbb N$-modules that is given by the (co)free dg $\mathbb S$-module functor composed with the forgetful functor.

\medskip

\subsection{Construction of the composition product} 
We denote by $u\operad A s$ the (graded) $\mathbb{N}$-module given by $u\operad A s(n) \coloneqq \kk$ (in degree $0$) for all $n \geq 0$. It has a obvious planar/non-symmetric operad structure, where all the composition maps are isomorphisms.
    
    \medskip
    
We denote by $u\operad A ss$ be the $\mathbb S$-module freely generated from $u\operad A s$, that is
    $$
    u\operad A ss \coloneqq u\operad A s \otimes \mathbb S~.
    $$
In particular, we have $u\operad A ss(n) =\kk[\mathbb S_n]$ for all natural integers $n \geq 0$.

\begin{lemma}
The underlying $\mathbb N$-module of $u\operad A ss$ has the structure of a planar operad.

\begin{enumerate}
\item The unit is given by the identity map of $\kk$.

\medskip

\item The composition is defined as follows. For every $(\sigma, \sigma_1, \ldots, \sigma_k) \in \mathbb S_k \times \mathbb S_{n_1} \times \cdots \times \mathbb S_{n_k}$ the composition of $\sigma \otimes \sigma_1 \otimes  \cdots \otimes  \sigma_k$ is given by the permutation 
\[
\begin{tikzcd}[column sep=6pc,row sep=2pc]
    \{1, \ldots, n\}
    \ar[d, "\cong"]
    \\
    \{1, \ldots, n_1\} \sqcup \cdots \sqcup \{1, \ldots, n_k\}
    \ar[d, "\sigma_1 \sqcup \cdots \sqcup \sigma_k"]
    \\
    \{1, \ldots, n_1\} \sqcup \cdots \sqcup \{1, \ldots, n_k\}
    \ar[d, "\sigma"]
    \\
    \{1, \ldots, n_{\sigma^{-1}(1)}\} \sqcup \cdots \sqcup \{1, \ldots, n_{\sigma^{-1}(k)}\}
    \ar[d, "\cong"]
    \\
    \{1, \ldots, n\}.
\end{tikzcd}
\]

in $u\operad A ss (n)$, where $n = n_1 + \cdots + n_k$.
\end{enumerate}
\end{lemma}

\begin{proof}
    This follows from a straightforward check.
\end{proof}

\begin{proposition}
The endofunctor $-\otimes \mathbb S$ of the category of dg $\mathbb N$-modules has a canonical structure of lax monoidal functor with respect to the planar composition product $\circ_{\pl}$. 

\begin{enumerate}
\item There is a canonical isomorphism $\operad I \cong \operad I \otimes \mathbb S$ given by the identity map,

\medskip

\item and a binatural map of dg $\mathbb{N}$-modules
\[
\varphi_{X,Y}: (X  \otimes \mathbb S) \comp_\pl (Y  \otimes \mathbb S) \longrightarrow (X \comp_\pl Y) \otimes \mathbb S
\]
given by
\[
\begin{tikzcd}[column sep=6pc,row sep=2pc]
    (X(k) \otimes \{\sigma\}) \otimes 
\left( 
(Y(n_1) \otimes \{\sigma_1\})
\otimes \cdots \otimes
(Y(n_k) \otimes \mathbb \{\sigma_k\})
\right)
\ar[d, "\cong"]
\\
X(k) \otimes 
\left( 
Y(n_1) 
\otimes \cdots \otimes
Y(n_k)
\right)
\ar[d, "\id \otimes \sigma"]
\\
X(k) \otimes 
\left( 
Y(n_{\sigma^{-1}(1)}) 
\otimes \cdots \otimes
Y(n_{\sigma^{-1}(k)}) 
\right)
\ar[d, "\cong"]
\\
X(k) \otimes 
\left( 
Y(n_{\sigma^{-1}(1)}) 
\otimes \cdots \otimes
Y(n_{\sigma^{-1}(k)}) 
\right)
\otimes \{\lambda\}
\end{tikzcd}
\]
where $\lambda$ is the composition in $u\operad A ss$ of $\sigma_k \otimes (\sigma_{n_1} \otimes \cdots \otimes \sigma_{n_k})$.
\end{enumerate}
\end{proposition}

\begin{proof}
This amounts to showing that the following diagrams
$$
\begin{tikzcd}[column sep=2.5pc,row sep=2.5pc]
    (X  \otimes \mathbb S) \comp_\pl (Y  \otimes \mathbb S) \comp_\pl (Z  
    \otimes \mathbb S)
    \ar[r, "\id \circ_\pl \varphi_{Y,Z}"] \ar[d,"\varphi_{X,Y} \circ_\pl \id",swap]
    & (X  \otimes \mathbb S) \comp_\pl ((Y \comp_\pl Z)  \otimes \mathbb S)
    \ar[d,"\varphi_{X,Y \circ_\pl Z}"]
    \\
    ((X \comp_\pl Y)  \otimes \mathbb S) \comp_\pl (Z  \otimes \mathbb S)
    \ar[r,"\varphi_{X \circ_\pl Y,Z}"]
    & (X \comp_\pl Y \comp_\pl Z)  \otimes \mathbb S~,
\end{tikzcd}
$$
and
$$
\begin{tikzcd}[column sep=2.5pc,row sep=2.5pc]
    (X  \otimes \mathbb S) \comp_\pl (\operad I  \otimes \mathbb S)
    \ar[r,"\varphi_{X,\operad I}"] \ar[d]
    & (X \comp_\pl \operad I)  \otimes \mathbb S
    \ar[d]
    \\
    X \otimes \mathbb S
    \ar[r, "\cong"]
    & X \otimes \mathbb S
    \\
    (\operad I  \otimes \mathbb S) \comp_\pl (X \otimes \mathbb S)
    \ar[r,"\varphi_{\operad I,X}"] \ar[u]
    & (\operad I \comp_\pl X)  \otimes \mathbb S
    \ar[u]
\end{tikzcd}
$$
are commutative for every dg $\mathbb N$-modules $X,Y,Z$. This follows from the associativity and the unitality of the compositions in the planar operad $u\operad A ss$.
\end{proof}

\begin{lemma}\label{lemmalaxmonoidalnt}
There are natural transformations
\[
\eta: \id_{\mathbb{N}} \longrightarrow - \otimes \mathbb S
\]

and 
\[
\mu: - \otimes \mathbb S \otimes \mathbb S \longrightarrow - \otimes \mathbb S
\]

of endofunctors in the category of dg $\mathbb N$-modules, induced, respectively, by the units and the multiplications in the group algebras $\kk[\mathbb{S}_n]$ for all $n \geq 0$. Furthermore, these natural transformations are monoidal natural transformations between lax monoidal functors.
\end{lemma}

\begin{proof}
The behaviour of the monad $- \otimes \mathbb S$ with respect to the unit $\operad I$ is straightforward to check. We check that the two following diagrams 
\[
    \begin{tikzcd}[column sep=3pc,row sep=3.5pc]
        (X \otimes \mathbb S \otimes \mathbb S)
        \comp_\pl (Y \otimes \mathbb S \otimes \mathbb S)
        \ar[r,"\varphi_{X \otimes \mathbb S,Y \otimes \mathbb S}"] \ar[d,"\mu_X \circ_\pl \mu_Y",swap]
        & ((X \otimes \mathbb S )
        \comp_\pl (Y \otimes \mathbb S )) \otimes \mathbb S
        \ar[r,"\varphi_{X,Y} \otimes \mathbb S"]
        & (X \comp_\pl Y) \otimes \mathbb S \otimes \mathbb S
        \ar[d,"\mu_{X \comp_\pl Y}"]
        \\
        (X \otimes \mathbb S )
        \comp_\pl (Y \otimes \mathbb S )
        \ar[rr,"\varphi_{X,Y}"]
        && 
        (X \comp_\pl Y) \otimes \mathbb S
    \end{tikzcd}
\]
and 
\[
    \begin{tikzcd}[column sep=2.5pc,row sep=2.5pc]
        (X \otimes \mathbb S ) \comp_\pl (Y \otimes \mathbb S )
        \ar[r,"\varphi_{X,Y}"]
        & (X \comp_\pl Y) \otimes \mathbb S \\
        X \comp_\pl Y \arrow[u,"\eta_X \circ_\pl \eta_Y"] \arrow[ru,"\eta_{X \comp_\pl Y}",swap]~,
        	&
    \end{tikzcd}
\]
are commutative for all dg $\mathbb{N}$-modules $X$ and $Y$. It follows from the compatibility of the compositions in the planar operad $u\operad A ss$ with the actions of the symmetric groups. 
\end{proof}

\begin{remark}
There is a natural transformation 
\[
\pi: \mathrm{U}(-) \otimes \mathbb S \longrightarrow \id_{\mathbb{S}}~,
\]
where $\mathrm{U}$ is the forgetful functor from dg $\mathbb{S}$-modules to dg $\mathbb{N}$-modules. For $M$ a dg $\mathbb S$-modules, $\pi_M: \mathrm{U}(M) \otimes \mathbb{S} \longrightarrow M$ is given by the dg $\mathbb{S}$-module structure of $M$. 
\end{remark}

\begin{definition}[Composition product]\label{definitioncomp}
Let $M, N$ be two dg $\mathbb S$-modules. We define the \textit{composition product} $M \comp N$ as the coequaliser of the reflexive pair of maps

\[
\begin{tikzcd}[column sep=5pc,row sep=4pc]
M \comp N \coloneqq \mathrm{Coeq}\Bigg(\displaystyle ((\mathrm{U}(M) \otimes \mathbb S) \comp_\pl (\mathrm{U}(N) \otimes \mathbb S))\otimes \mathbb S \arrow[r,"\mu_{\mathrm{U}(M) \comp_\pl \mathrm{U}(N)}~\varphi_{\mathrm{U}(M),\mathrm{U}(N)} \otimes \mathbb S",shift right=1.1ex,swap]  \arrow[r,"(\pi_M \comp_\pl \pi_N) \otimes \mathbb{S}"{name=SD},shift left=1.1ex ]
&( \mathrm{U}(M) \comp_\pl  \mathrm{U}(N)) \otimes \mathbb S\Bigg)~,
\end{tikzcd}
\]
\vspace{0.1pc}
    
in the category of dg $\mathbb S$-modules.
\end{definition}

For any dg $\mathbb S$-modules $M, N$, there is, by definition, a natural map 
\[
\rho_{M,N}: (\mathrm{U}(M) \comp_\pl  \mathrm{U}(N)) \otimes \mathbb S \longrightarrow M \circ N~,
\]
given by the universal property of the coequalizer. 

\begin{proposition}\label{lemmacompositioniso}
For every two dg $\mathbb N$-modules $X,Y$, the natural map 
\[
\begin{tikzcd}[column sep=3.5pc,row sep=4pc]
\psi_{X,Y}: (X \comp_\pl Y) \otimes \mathbb S \arrow[r,"(\eta_X \comp_\pl \eta_Y) \otimes \mathbb{S}"]
&(X \otimes \mathbb S \comp_\pl Y \otimes \mathbb S) \otimes \mathbb S \arrow[r,"\rho_{X \otimes \mathbb S ,Y \otimes \mathbb S}"]
&(X \otimes \mathbb S) \comp (Y\otimes \mathbb S)
\end{tikzcd}
\]
is an isomorphism of dg $\mathbb S$-modules. 
\end{proposition}

\begin{proof}
    Let us consider the composite map  
    \[
    \mu_{X \comp_\pl Y}~(\varphi_{X,Y} \otimes \mathbb S): ((X \otimes \mathbb S) \comp_\pl (Y\otimes \mathbb S)) \otimes \mathbb S
    \longrightarrow (X \comp_\pl Y) \otimes \mathbb S \otimes \mathbb S
    \longrightarrow (X \comp_\pl Y) \otimes \mathbb S~.
    \]
    Using Lemma \ref{lemmalaxmonoidalnt} in a straightforward way, one can prove that it factors through $(X \otimes \mathbb S) \comp (Y\otimes \mathbb S)$ and that the composite map
    \[
    \mu_{X \comp_\pl Y}~(\varphi_{X,Y} \otimes \mathbb S)~\psi_{X,Y}: (X \comp_\pl Y) \otimes \mathbb S
    \longrightarrow (X \otimes \mathbb S) \comp (Y\otimes \mathbb S)
    \longrightarrow (X \comp_\pl Y) \otimes \mathbb S
    \]
    is the identity. To prove that the composite map
    \[
    \psi_{X,Y}~\mu_{X \comp_\pl Y}~(\varphi_{X,Y} \otimes \mathbb S):(X \otimes \mathbb S) \comp (Y\otimes \mathbb S)
    \longrightarrow (X \comp_\pl Y) \otimes \mathbb S
    \longrightarrow (X \otimes \mathbb S) \comp (Y\otimes \mathbb S)
    \]
    is the identity, it is enough to prove that the composition
    \[
    ((X \otimes \mathbb S) \comp_\pl (Y\otimes \mathbb S)) \otimes \mathbb S
    \longrightarrow
    (X \otimes \mathbb S) \comp (Y\otimes \mathbb S)
    \longrightarrow (X \comp_\pl Y) \otimes \mathbb S
    \longrightarrow (X \otimes \mathbb S) \comp (Y\otimes \mathbb S)
    \]
    given by $\psi_{X,Y}~\mu_{X \comp_\pl Y}~(\varphi_{X,Y} \otimes \mathbb S)~\rho_{X \otimes \mathbb S ,Y \otimes \mathbb S}$ 
    is equal to the canonical quotient map $\rho_{X \otimes \mathbb S ,Y \otimes \mathbb S}$. Using 
    Lemma \ref{lemmalaxmonoidalnt} again, we can rewrite this composition as:
   	\[
   	\begin{tikzcd}
   	((X \otimes \mathbb S) \comp_\pl (Y\otimes \mathbb S)) \otimes \mathbb S \arrow[d,"(\eta_{X \otimes \mathbb S} \comp_\pl \eta_{Y \otimes \mathbb S}) \otimes \mathbb{S}"] \\
   	((X \otimes \mathbb S \otimes \mathbb S) \comp_\pl (Y\otimes \mathbb S \otimes \mathbb S)) \otimes \mathbb S \arrow[d,"\mu_{(X \otimes \mathbb{S})\comp_\pl (Y \otimes \mathbb{S})}~(\varphi_{X \otimes \mathbb{S},Y \otimes \mathbb{S}} \otimes \mathbb{S})"] \\
   	((X \otimes \mathbb S) \comp_\pl (Y\otimes \mathbb S)) \otimes \mathbb S \arrow[d,"\rho_{X \otimes \mathbb S ,Y \otimes \mathbb S}"]\\
    (X \otimes \mathbb S) \comp (Y\otimes \mathbb S)~.
   	\end{tikzcd}
   	\]
The universal property that defines the composition product $- \comp -$ allows us to change the second map $\mu_{(X \otimes \mathbb{S})\comp_\pl (Y \otimes \mathbb{S})} ~(\varphi_{X \otimes \mathbb{S},Y \otimes \mathbb{S}} \otimes \mathbb{S})$ by $(\mu_X \comp_\pl \mu_Y ) \otimes \mathbb{S}$ without modifying the whole composition. Then, the composition of the two first arrows is the identity, which implies that the whole composition is equal to its last component $\rho_{X \otimes \mathbb S ,Y \otimes \mathbb S}$, which concludes the proof.
\end{proof}

\begin{theorem}\label{prop: equal to the usual composition product}
Let $M, N$ be two dg $\mathbb S$-modules. There is a natural isomorphism

\[
M \comp N(n) \cong \bigoplus_{k\geq 0} M(k) \otimes_{\mathbb{S}_k} \left( \bigoplus_{i_1 + \cdots + i_k = n} \mathrm{Ind}_{\mathbb{S}_{i_1} \times \cdots \times \mathbb{S}_{i_k}}^{\mathbb{S}_n} (N(i_1) \otimes \cdots \otimes N(i_k))\right)
\]
\vspace{0.1pc}

of dg $\mathbb S$-modules, where $\mathrm{Ind}$ stands for the induced representation.
\end{theorem}

\begin{proof}
By Proposition \ref{lemmacompositioniso}, both expressions are isomorphic for free dg $\mathbb{S}$-modules. Furthermore, both expressions commute with sifted colimits. Hence both expressions are isomorphic for any pair of dg $\mathbb{S}$-modules.
\end{proof}

\begin{remark}
Therefore Definition \ref{definitioncomp} coincides with the standard definition of the composition product of dg $\mathbb{S}$-modules present in the literature. See for instance \cite{LodayVallette}. 
\end{remark}

\begin{proposition}
The composition product $\comp$ defines the structure of a monoidal category on dg $\mathbb S$-module whose unit is given by $\operad I$.
\end{proposition}

\begin{proof}
Let $X,Y,Z$ be three dg $\mathbb{N}$-modules. One can construct an associator isomorphism
\[
\begin{tikzcd}
\alpha_{X,Y,Z}: ((X \otimes \mathbb S) \comp (Y \otimes \mathbb S)) \comp (Z \otimes \mathbb S) 
\ar[r, "\cong"]
&(X \otimes \mathbb S) \comp ((Y \otimes \mathbb S) \comp (Z \otimes \mathbb S))~,
\end{tikzcd}
\]
using the binatural isomorphism $\psi$ of Proposition \ref{lemmacompositioniso} and the associator isomorphism of the composition product $\comp_\pl$ in the category of dg $\mathbb{N}$-modules. Furthermore, this associator isomorphism commutes with the natural transformation $\mu: - \otimes \mathbb S \otimes \mathbb S \longrightarrow - \otimes \mathbb S$, hence it induces an associator isomorphism for all free dg $\mathbb{S}$-modules. Using the compatibility of these constructions with sifted colimits, it induces a associator isomorphism for any three dg $\mathbb{S}$-modules. Likewise for the unitor isomorphisms. The commutativity of the required diagrams then follows from the commutativity of the unitors and the associators of the planar composition product $- \comp_\pl -$.
\end{proof}

\begin{proposition}
The functor $-\otimes \mathbb S$ from dg $\mathbb N$-modules to dg $\mathbb S$-modules is strong monoidal.
\end{proposition}

\begin{proof}
This structure of a strong monoidal functor is given by the identity map $\operad I \longrightarrow \operad I \otimes \mathbb S$ and the binatural map $\psi$ of Proposition \ref{lemmacompositioniso}. The fact that these maps define the structure of a strong monoidal functor is a direct consequence of the fact that they are used in the definition of the associator and the unitors related to the composition product.
\end{proof}

\subsection{Operads and cooperads}

\begin{definition}[dg operad]
A \textit{dg operad} $\operad {P}$ amounts to the data of a monoid $(\mathcal{P},\gamma,\eta)$ in the category of dg $\mathbb N$-modules with respect to the composition product. 
\end{definition}

\begin{definition}[Augmented dg operad]
An \textit{augmented dg operad} $\operad P$ amounts to the data of a dg operad $(\mathcal{P},\gamma,\eta)$ equipped with a morphism of dg operads $\nu: \operad P \longrightarrow \operad I$ such that $\nu \circ \eta = \mathrm{id}.$
\end{definition}

Given an augmented planar dg operad $\operad P$, we will denote by $\overline{\operad P}$ the kernel of the augmentation map. 

\begin{definition}[dg cooperad]
A \textit{dg cooperad} $\operad C$ amounts to the data of a comonoid $(\C, \Delta, \epsilon)$ in the category of dg $\mathbb S$-modules with respect to the composition product. 
\end{definition}

Given a dg cooperad $\C$, we will denote by $\overline{\operad C}$ the kernel of the counit map. 

\begin{definition}[Coaugmented dg cooperad]
A \textit{coaugmented dg cooperad} $\operad C$ amounts to the data of a dg cooperad $(\C, \Delta, \epsilon)$ equipped together with a morphism of planar dg cooperads $\mu: \operad I \longrightarrow \operad C$ such that $\epsilon \circ \mu = \mathrm{id}$. 
\end{definition}

The strong monoidal structure on the functor $-\otimes \mathbb S$ yields two adjunctions 

\[
\begin{tikzcd}[column sep=5pc,row sep=3pc]
          \dgoperads_\pl  \arrow[r, shift left=1.1ex, "-\otimes \mathbb S"{name=F}] &\dgoperads , \arrow[l, shift left=.75ex, "\mathrm{U}_{\mathbb S}"{name=U}]
            \arrow[phantom, from=F, to=U, , "\dashv" rotate=-90]
\end{tikzcd}
\]
\[
\begin{tikzcd}[column sep=5pc,row sep=3pc]
          \dgcooperads_\pl  \arrow[r, shift left=1.1ex, "-\otimes \mathbb S"{name=F}] &\dgcooperads, \arrow[l, shift left=.75ex, "\mathrm{U}_{\mathbb S}"{name=U}]
            \arrow[phantom, from=F, to=U, , "\dashv" rotate=90]
\end{tikzcd}
\]

that lift, respectively, the adjunction $-\otimes \mathbb S \dashv U_{\mathbb S}$ and the adjunction $U_{\mathbb S} \dashv -\otimes \mathbb S$ that relate dg $\mathbb S$-modules to dg $\mathbb N$-modules.

\medskip

\begin{remark}
The adjunction $- \otimes \mathbb S \dashv U_{\mathbb S}$ relating dg operads to planar dg operads is monadic since its right adjoint preserves coreflexive equalisers and is conservative. However the other adjunction $U_{\mathbb S} \dashv - \otimes \mathbb S \dashv $ is not a priori comonadic.
\end{remark}


\section{The tree monad}

\subsection{The tree intertwining map}
Let us consider the following sequence of adjunctions

\[
\begin{tikzcd}[column sep=3.5pc,row sep=0.5pc]
          \catdgmod{\mathbb N}  \arrow[r, shift left=1.1ex, "\treemod_\pl"{name=F}] & \dgoperads_\pl  \arrow[r, shift left=1.1ex, "- \otimes \mathbb S"{name=A}] \arrow[l, shift left=.75ex, "\mathrm{U}"{name=U}]
          & \dgoperads~. \arrow[l, shift left=.75ex, "\mathrm{U}"{name=B}]
            \arrow[phantom, from=F, to=U, , "\dashv" rotate=-90]
             \arrow[phantom, from=A, to=B, , "\dashv" rotate=-90]
\end{tikzcd}
\]
The composite right adjoint is conservative and preserves sifted colimits, since both of its components do so. Thus, it is monadic. The related monad is given by $\treemod_\pl(-) \otimes \mathbb S$. Its monad structure is related to the monad structure of $\treemod_\pl$ by an exchange map $\treemod_\pl(-\otimes \mathbb S) \longrightarrow \treemod_\pl(-) \otimes \mathbb S$ defined below.

\begin{definition}[The tree intertwining map]
The \textit{tree intertwining map} $\tau$ is the natural morphism of dg $\mathbb N$-modules
\[
\tau: \treemod_\pl(-\otimes \mathbb S) \longrightarrow \treemod_\pl(\treemod_\pl(-) \otimes \mathbb S) \otimes \mathbb S
\longrightarrow \treemod_\pl(-) \otimes \mathbb S.
\]
whose first part is given by the units of the monads $- \otimes \mathbb S$ and $\treemod_\pl$ and whose second part is the monad product of $\treemod_\pl(-) \otimes \mathbb S$.
\end{definition}

The tree intertwining map $\tau$ is a morphism of planar dg operads. Hence the following natural diagram 
\[
\begin{tikzcd}[column sep=2.5pc,row sep=2.5pc]
    (X \otimes \mathbb S) \comp_\pl \treemod_\pl (X \otimes \mathbb S)
    \ar[r, "\cong"] \ar[d, "\id \comp_\pl \tau"']
    & \overline{\treemod}_\pl (X \otimes \mathbb S)
    \ar[r, hookrightarrow] \ar[dd, "\tau"]
    & {\treemod}_\pl (X \otimes \mathbb S)
    \ar[dd, "\tau"]
    \\
    (X \otimes \mathbb S) \comp_\pl (\treemod_\pl (X) \otimes \mathbb S)
    \ar[d]
    \\
    (X \comp_\pl \treemod_\pl (X)) \otimes \mathbb S
    \ar[r, "\cong"]
    & \overline{\treemod}_\pl (X) \otimes \mathbb S
    \ar[r, hookrightarrow]
    & {\treemod}_\pl (X) \otimes \mathbb S
\end{tikzcd}
\]
commutes and one can describe $\tau$ by induction on the size of trees.

\medskip

We denote by $\categ{Trees}$  the groupoid of trees, whose objects are planar trees and whose morphisms are isomorphisms which are not necessarily planar. We refer to Appendix \ref{appendixtrees} for more details. 

\begin{definition}[Planar endofunctor of a sub-groupoid]
Let $g$ denote a full sub-groupoid of the groupoid of trees $\categ{Trees}$. Let us assume that $g$ is \textit{connected} within $\categ{Trees}$ in the sense that, for any isomorphism of trees $t \longrightarrow t'$, $t'$ belongs to $g$ whenever $t$ does. The \textit{planar endofunctor} $g_\pl(-)$ associated to $g$ is the endofunctor of dg $\mathbb N$-modules given, for a dg $\mathbb{N}$-module $X$, by
\[
g_\pl(X) \coloneqq \bigoplus_{t \in g} t(X)~,
\]
where the sum is taken over planar trees in $g$ up to planar isomorphisms. 
\end{definition}

\begin{proposition}
For every connected full sub-groupoid $g$ of $\categ{Trees}$, the tree intertwining map 
\[
\tau: \treemod_\pl (- \otimes \mathbb S) \longrightarrow \treemod_\pl (-) \otimes \mathbb S
\]
restricts to a morphism
\[
\tau|_g: g_\pl (- \otimes \mathbb S) \longrightarrow g_\pl (-) \otimes \mathbb S.
\]
of dg $\mathbb{N}$-modules.
\end{proposition}

\begin{proof}
Let $X$ be a dg $\mathbb{N}$-module. It is suffices to prove that for every planar tree $t$, the map
\[
t(X \otimes \mathbb S) \hookrightarrow \treemod_\pl (X \otimes \mathbb S) \longrightarrow \treemod_\pl (X) \otimes \mathbb S
\]
factors through $g_\pl (-) \otimes \mathbb S$ where $g$ is the connected component of the groupoid of trees that contains $t$.
This can be done by induction on the height of trees. On the trivial tree, the tree intertwining map is given by the identity map of $\operad I \otimes \mathbb S$. On trees of height one, the tree intertwining map is given by the identity of $X \otimes \mathbb S$.

\medskip

Let us assume that it is satisfied for all tree of height equal or lower than $m$ for some natural integer $m \geq 1$. Let $t$ be a planar tree with $n$ leaves that is made up of a root node with $k \geq 1$ inputs and topped by $k$ subtrees $t_1, \ldots, t_k$ of height equal or lower than $m$ and with $n_1,\ldots, n_k$ leaves. The height of $t$ is thus lower or equal to $m+1$. Let us denote $g_i$ the connected component within the groupoid of tree of $t_i$ for $i \in \{1, \ldots, k\}$ and $g$ the connected component of $t$. On $t$, the intertwining map is the composition
$$
\begin{tikzcd}
t(X \otimes \mathbb S)(n) \coloneqq (X(k) \otimes \kk[\mathbb S_{k}]) \otimes t_1(X \otimes \mathbb S)(n_1)
 \otimes \cdots \otimes t_k(X \otimes \mathbb S)(n_k)
 \ar[d]
 \\
 (X(k) \otimes \kk[\mathbb S_{k}])
 \otimes (g_{1,\pl}(X)(n_1) \otimes \kk[\mathbb S_{n_1}])
 \otimes \cdots \otimes (g_{k,\pl}(X)(n_k) \otimes \kk[\mathbb S_{n_k}])
 \ar[d, "\cong"]
 \\
 \displaystyle \bigoplus_{t'_1 \in g_1} \cdots \bigoplus_{t'_k \in g_k}
 (X(k) \otimes \kk[\mathbb S_{k}])
 \otimes (t'_{1}(X)(n_1) \otimes \kk[\mathbb S_{n_1}])
 \otimes \cdots \otimes (t'_{k}(X)(n_k) \otimes \kk[\mathbb S_{n_k}])
 \ar[d,"\xi"]
 \\
  \displaystyle \bigoplus_{t'_1 \in g_1} \cdots \bigoplus_{t'_k \in g_k} \bigoplus_{\sigma \in \mathbb S_{k}}
 X(k) \otimes t'_{\sigma^{-1}(1)}(X)(n_{\sigma^{-1}(1)}) \otimes \cdots \otimes t'_{\sigma^{-1}(k)}(X)(n_{\sigma^{-1}(k)}) \otimes \kk[\mathbb S_{n}]
 \ar[d, "\cong"]
 \\
  \displaystyle \bigoplus_{t'_1 \in g_1} \cdots \bigoplus_{t'_k \in g_k} \bigoplus_{\sigma \in \mathbb S_{k}}
 \mathrm{graft}_k(t'_{\sigma^{-1}(1)} ,\ldots, t'_{\sigma^{-1}(k)}) (X)(n) \otimes \kk[\mathbb S_{n}]
 \ar[d, hookrightarrow]
 \\
 \treemod_\pl (X) \otimes \mathbb S~.
\end{tikzcd}
$$
Here the third map $\xi$ is given by
\[
\xi: X(k) \otimes \sigma) 
 \otimes (t'_1(X) \otimes \sigma_{1})
 \otimes \cdots \otimes (t'_k(X) \otimes \sigma_k)
 \mapsto
 X(k) \otimes t'_{\sigma^{-1}(1)}(X) \otimes \cdots \otimes t'_{\sigma^{-1}(k)} \otimes \lambda~,
\]
where $\lambda \in \kk[\mathbb S_n]$ is the composition $\sigma \comp (\sigma_1 \otimes \cdots \otimes \sigma_k)$
in the operad $u\operad A ss$. Moreover, for every planar trees $t'_1\in g_1, \ldots, t'_k \in g_k$, $\mathrm{graft}_k(t'_{\sigma^{-1}(1)} ,\ldots, t'_{\sigma^{-1}(k)})$ is the planar tree with a root node with $k$ inputs and toppped by the sub-trees, in order, $t'_1, \ldots, t'_k$. Therefore it is clear that it is isomorphic to $t$. Thus the map $t(X \otimes \mathbb S) \hookrightarrow \treemod_\pl(X \otimes \mathbb S) \longrightarrow  \treemod_\pl(X) \otimes \mathbb S$ factors through $g_\pl(X) \otimes \mathbb S$ and we can conclude by induction.
\end{proof}

\subsection{The tree monad}
Let us consider the following commutative square of forgetful functors
\[
\begin{tikzcd}[column sep=3pc,row sep=3pc]
    \dgoperads
    \ar[d, "\mathrm{U}_{\mathbb S}"']
    \ar[r,"\mathrm{U}_{\mathrm{Op}}"]
    & \catdgmod{\mathbb S}
    \ar[d, "\mathrm{U}"]
    \\
    \dgoperads_\pl
    \ar[r,"\mathrm{U}_{\mathrm{nsOp}}"]
    &\catdgmod{\mathbb N}~.
\end{tikzcd}
\]
The two vertical functors $\mathrm{U}$ and $\mathrm{U}_{\mathbb S}$ are monadic and create sifted colimits. Since the bottom horizontal forgetful functor $\mathrm{U}_{\mathrm{nsOp}}$ admits a left adjoint, we can use the adjoint lifiting theorem, see \cite[Appendix A]{premierpapier} to deduce that the top horizontal functor $\mathrm{U}_{\mathrm{Op}}$ also admits a left adjoint. Moreover, $\mathrm{U}_{\mathrm{Op}}$ is conservative and preserves and creates sifted colimits. Therefore it is monadic.

\begin{definition}[The tree monad]
We define the \textit{tree monad} $\treemod$ to be the left adjoint of the forgetful functor $\mathrm{U}_{\mathrm{Op}}$, fitting in the following  adjunction

\[
\begin{tikzcd}[column sep=3.5pc,row sep=0.5pc]
          \catdgmod{\mathbb S}  \arrow[r, shift left=1.1ex, "\treemod"{name=F}] &\dgoperads~,   \arrow[l, shift left=.75ex, "\mathrm{U}_{\mathrm{Op}}"{name=U}]
            \arrow[phantom, from=F, to=U, , "\dashv" rotate=-90]        
\end{tikzcd}
\]
which is furthermore monadic.
\end{definition}

Again, by the adjunction lifting theorem, we know that given a dg $\mathbb{S}$-module $M$, then $\treemod(M)$ is given by a reflexive coequalizer of the form
\[
\begin{tikzcd}[column sep=3pc,row sep=4pc]
\mathrm{Coeq}\Bigg(\displaystyle (\treemod_\pl((\mathrm{U}(M) \otimes \mathbb S) \otimes \mathbb S \arrow[r,"",shift right=1.1ex,swap]  \arrow[r,""{name=SD},shift left=1.1ex ]
&\treemod_\pl(\mathrm{U}(M)) \otimes \mathbb S\Bigg) \arrow[r,dashed]
&\treemod (M)~,
\end{tikzcd}
\]
where one of the maps is build from the dg $\mathbb{S}$-module structure of $M$ and the other using the monad structures $- \otimes \mathbb S$.

\medskip 

\begin{lemma}
Let $X$ be a dg $\mathbb{N}$-module. The canonical map 
\[
\nu_X: \treemod (X \otimes \mathbb S) \longrightarrow \treemod_\pl (X) \otimes \mathbb S
\]
is a isomorphism of dg operads, natural in $X$. 
\end{lemma}

\begin{proof}
Follows from the definitions.
\end{proof}

\begin{definition}[Endofunctor of a sub-groupoid]
Let $g$ be a connected full sub-groupoid of the groupoid of trees $\categ{Trees}$. The \textit{endofunctor} $g(-)$ associated to $g$ is the endofunctor of the category of dg $\mathbb{S}$-modules given, for a dg $\mathbb{S}$-module $M$, by the following coequalizer
\[
\begin{tikzcd}[column sep=3pc,row sep=4pc]
\mathrm{Coeq}\Bigg(\displaystyle (g_\pl(\mathrm{U}(M) \otimes \mathbb S) \otimes \mathbb S \arrow[r,"",shift right=1.1ex,swap]  \arrow[r,""{name=SD},shift left=1.1ex ]
&g_\pl(\mathrm{U}(M)) \otimes \mathbb S\Bigg) \arrow[r,dashed]
&g(M)~,
\end{tikzcd}
\]
where one map is given by the dg $\mathbb{S}$-module structure of $M$ and the other by the monad structure of $- \otimes \mathbb S$.
\end{definition}

\begin{remark}
Let $M$ be a dg $\mathbb{S}$-module, the canonical map 
\[
\bigoplus_g g(X) \longrightarrow \treemod (X)
\]
is an isomorphism when the sum is over every connected components of the groupoid $\categ{Trees}$.
\end{remark}

\begin{notation}
Let $n$ be a natural integer. For some particular sub-groupoids $g$, we use the following notations.

\medskip

\begin{itemize}
\item For $g$ being the sub-groupoid of non trivial trees, we denote $g(X)$ as $\overline{\treemod}(X)$, and call it the \textit{reduced tree} endofunctor.

\medskip

    \item For $g$ the sub-groupoid of trees of height at most $n$, we denote $g(X)$ as $\treemod_{\leq n}(X)$.
    
\medskip

    \item For $g$ the sub-groupoid of trees with $n$ nodes (resp. at most $n$ nodes), we denote $g(X)$ as $\treemod^{(n)}(X)$.
    
\medskip

	\item For $g$ the sub-groupoid of trees with at most $n$ nodes, we denote $g(X)$ as $\treemod^{(\leq n)}(X)$.

\end{itemize}
\end{notation}

\begin{proposition}\label{corollarygiso}
    Let $X$ be a dg $\mathbb{N}$-module. The natural isomorphism 
    \[
    \nu_X: \treemod (X \otimes \mathbb S) \longrightarrow \treemod_\pl (X) \otimes \mathbb S
    \]
    restrict to a natural isomorphism
    \[
    \nu_X|_g: g(X \otimes \mathbb S) \longrightarrow g_\pl (X) \otimes \mathbb S
    \]
    for every connected full sub-groupoid $g$ of the groupoid of trees $\categ{Trees}$.
\end{proposition}

\begin{proof}
It suffices to notice that if $\overline{g}$ is the complement sub-groupoid of $g$ in $\categ{Trees}$, then the map
\[
\nu_X: \treemod (X \otimes \mathbb S) \longrightarrow \treemod_\pl (X) \otimes \mathbb S
\]
rewrites as the sum
    \[
    \nu_X|_g \oplus \nu_X|_{\overline g} :
    g (X \otimes \mathbb S) \oplus \overline g (X \otimes \mathbb S)
    \longrightarrow g_\pl (X) \otimes \mathbb S \oplus \overline g_\pl (X) \otimes \mathbb S~.
    \]
    Since $\nu_X$ is an isomorphism, so are $\nu_X|_g$ and $\nu_X|_{\overline g}$.
\end{proof}

\begin{proposition}\label{corollarytreemodcommutation}
Let $g$ be a connected full sub-groupoid of of the groupoid of trees $\categ{Trees}$. Let $M$ be a dg $\mathbb{S}$-module. The isomorphism 
\[
\nu_{\mathrm{U}(M)}|_g: g(\mathrm{U}(M) \otimes \mathbb S) \longrightarrow g_\pl (\mathrm{U}(M)) \otimes \mathbb S
\]
send the natural diagram in dg $\mathbb S$-modules
\[
\begin{tikzcd}[column sep=3pc,row sep=4pc]
\mathrm{Coeq}\Bigg(\displaystyle g_\pl(\mathrm{U}(M) \otimes \mathbb S) \otimes \mathbb S \arrow[r,"",shift right=1.1ex,swap]  \arrow[r,""{name=SD},shift left=1.1ex ]
&g_\pl(\mathrm{U}(M)) \otimes \mathbb S\Bigg) \arrow[r,dashed]
&g(M)~,
\end{tikzcd}
\]

to the natural diagram
\[
\begin{tikzcd}[column sep=3pc,row sep=4pc]
\mathrm{Coeq}\Bigg(\displaystyle g(\mathrm{U}(M) \otimes \mathbb S \otimes \mathbb S) \otimes \mathbb S \arrow[r,"",shift right=1.1ex,swap]  \arrow[r,""{name=SD},shift left=1.1ex ]
&g(\mathrm{U}(M) \otimes \mathbb S) \otimes \mathbb S \Bigg) \arrow[r,dashed]
&g(M)~,
\end{tikzcd}
\]
whose the maps are given by the monad structure on $- \otimes S$ and the action of this monad on $M$.
\end{proposition}

\begin{proof}
It suffices to prove the result for $g$ the whole groupoid of trees $\categ{Trees}$. This amounts to prove that the following diagrams 
    $$
    \begin{tikzcd}[column sep=2.5pc,row sep=2.5pc]
        \treemod (\mathrm{U}(M) \otimes \mathbb S \otimes \mathbb S)
        \ar[r,"\treemod(\mu_{\treemod})"] \ar[d,"\nu_{\mathrm{U}(M) \otimes \mathbb S}",swap]
        & \treemod (\mathrm{U}(M) \otimes \mathbb S)
        \ar[d,"\nu_{\mathrm{U}(M)}"]
        \\
        \treemod_\pl (\mathrm{U}(M) \otimes \mathbb S) \otimes \mathbb S
        \ar[r,"\tau_{\mathrm{U}(M)}"]
        & \treemod_\pl (\mathrm{U}(M)) \otimes \mathbb S
    \end{tikzcd}
    \begin{tikzcd}[column sep=2.5pc,row sep=2.5pc]
        \treemod (\mathrm{U}(M) \otimes \mathbb S)
        \ar[r,"\nu_{\mathrm{U}(M)}"] \ar[rd,"\pi_M ",swap]
        & \treemod_\pl (\mathrm{U}(M)) \otimes \mathbb S
        \ar[d,two heads]
        \\
        & \treemod (\mathrm{U}(M))
    \end{tikzcd}
    $$
are commutative in the category of dg operads. Here the top horizontal map of the left square diagram $\treemod(\mu_{\mathrm{U}(M)})$ is induced by the monad product $\mu: - \otimes \mathbb S \otimes \mathbb S \longrightarrow -\otimes \mathbb S$ and the bottom horizontal map is the tree intertwining map $\tau$. This commutativity follows from the fact that $\treemod(-)$ is the free dg operad functor.
\end{proof}

\begin{corollary}\label{corollary: iso de la récurrence sur les arbres}
Let $n$ be a natural integer and $M$ a dg $\mathbb{S}$-module. The composition of the free operad $\gamma_M: \treemod \comp \treemod M \longrightarrow \treemod M$ induces the following isomorphisms

    \begin{align*}
        &M \comp \treemod M \cong \overline{\treemod} M~,
        \\
        &M \comp \treemod_{\leq n} M \cong \overline{\treemod}_{\leq n+1} M~,
    \end{align*}

    of dg $\mathbb{S}$-modules. 
\end{corollary}

\begin{proof}
First, let us prove the fact that the composition restricts to the isomorphism $M \comp \treemod M \cong \overline{\treemod} M$. The fact that it restricts to the other isomorphisms $M \comp \treemod_{\leq n} M \cong \overline{\treemod}_{\leq n+1} M$ will follow from analogue arguments.
    
\medskip

When $M$ is a free dg $\mathbb S$-module of the form $X \otimes \mathbb S$, this composition rewrites as 
    \[
    (\treemod_\pl (X) \comp_\pl \treemod_\pl (X)) \otimes \mathbb S \longrightarrow \treemod_\pl (X) \otimes \mathbb S
    \]
    which restricts to an isomorphism
    \[
    (X \comp_\pl \treemod_\pl X)\otimes \mathbb S \cong \overline{\treemod}_\pl (X) \otimes \mathbb S.   
    \]
    The general case follows from the identification of the subsequent coequalizers that define each of the terms.
\end{proof}


\section{The reduced tree comonad and conilpotent cooperads}
We define conilpotent dg cooperad as coalgebras over the reduced tree comonad. Once again, this corresponds to the fact that any iteration of partial decompositions in the cooperad is eventually trivial. First, we endow the reduced tree endofunctor with a comonad structure, induced by the comonad structure of the reduced planar tree endofunctor.

\subsection{The reduced tree comonad}
\begin{lemma}\label{lemmaconilpotentcooperadlemmaone}
Let $X$ be a dg $\mathbb{N}$-module. The tree intertwining map $\tau$ is compatible with the planar cooperad structure of the reduced planar tree comonad. That is, the following diagram 
    \[
    \begin{tikzcd}[column sep=2.5pc,row sep=2.5pc]
        \overline{\treemod}_\pl(X \otimes \mathbb S)
        \ar[rr, "\tau_{X}"] \ar[d,"\Delta",swap]
        && \overline{\treemod}_\pl(X) \otimes \mathbb S 
        \ar[d,"\Delta \otimes \mathbb{S}"]
        \\
        \overline{\treemod}_\pl(X \otimes \mathbb S) \comp_\pl \overline{\treemod}_\pl(X \otimes \mathbb S)
        \ar[r,"\tau_{X} \comp_\pl \tau_{X}"]
        & (\overline{\treemod}_\pl(X) \otimes \mathbb S) \comp_\pl (\overline{\treemod}_\pl(X) \otimes \mathbb S)
        \ar[r,"\varphi"]
        & (\overline{\treemod}_\pl(X)  \comp_\pl \overline{\treemod}_\pl(X)) \otimes \mathbb S,
    \end{tikzcd}
   	\]
commutes in the category of dg $\mathbb{N}$-modules. Here $\Delta$ is the planar cooperad structure of $\overline{\treemod}_\pl(X \otimes \mathbb S)$ the horizontal map $\varphi$ is the lax monoidal functor structure on $- \otimes \mathbb S$.
\end{lemma}

\begin{proof}
Let us prove this result by induction on the height of the trees. This amounts to show that for every natural integer $n \geq 1$, the pre-compositions by the inclusion $\overline{\treemod}_{\pl, \leq n}(X \otimes \mathbb S) \hookrightarrow\overline{\treemod}_{\pl}(X \otimes \mathbb S)$ of the arrows in the square diagram are equal.

\medskip

For trees of height $n=1$, the diagram above restricts to
    \[
    \begin{tikzcd}[column sep=2.5pc,row sep=2.5pc]
        X \otimes \mathbb S
        \ar[r, "\id"] \ar[d, "\cong"']
        & X \otimes \mathbb S
        \ar[d, "\cong"]
        \\
        (X \otimes \mathbb S) \comp_\pl \operad I
        \ar[r]
        &  (X \comp_\pl \operad I) \otimes \mathbb S~,
    \end{tikzcd}
    \]
    which commutes. Let us assume that the result is satisfied up to a natural integer $n \geq 1$. Now, let us consider the following diagram
    $$
    \begin{tikzcd}[column sep=2.5pc,row sep=2.5pc]
        \overline{\treemod}_{\pl, \leq n+1}(X \otimes \mathbb S)
        \ar[r,"\Delta_{\leq n+1}"] \ar[d, "\cong"']
        & \overline{\treemod}_{\pl, \leq n+1} (X \otimes \mathbb S) \comp_\pl {\treemod}_{\pl, \leq n}(X \otimes \mathbb S)
        \ar[d, "\cong"]
        \\
        (X \otimes \mathbb S) \comp_\pl {\treemod}_{\pl, \leq n}(X\otimes \mathbb S)
        \ar[r,"\id \circ_\pl ~ \Delta_{\leq n}"] \ar[dd,"\id \circ_\pl~ (\tau_X)_{\leq n}",swap]
        & (X \otimes \mathbb S) \comp_\pl {\treemod}_{\pl, \leq n}(X\otimes \mathbb S) \comp_\pl {\treemod}_{\pl, \leq n}(X\otimes \mathbb S)
        \ar[d,"\id \circ_\pl~ (\tau_X)_{\leq n} \circ_\pl~ (\tau_X)_{\leq n}"]
        \\
        & (X \otimes \mathbb S) \comp_\pl ({\treemod}_{\pl, \leq n}(X)\otimes \mathbb S) \comp_\pl ({\treemod}_{\pl, \leq n}(X)\otimes \mathbb S)
        \ar[d,"\id \circ_\pl~ \varphi"]
        \\
        (X \otimes \mathbb S) \comp_\pl ({\treemod}_{\pl, \leq n}(X)\otimes \mathbb S)
        \ar[r,"\id \circ_\pl ~ \Delta_{\leq n}"] \ar[d,"\varphi",swap]
        & (X \otimes \mathbb S) \comp_\pl (({\treemod}_{\pl, \leq n}(X) \comp_\pl {\treemod}_{\pl, \leq n}(X))\otimes \mathbb S)
        \ar[d,"\varphi"]
        \\
        (X \comp_\pl {\treemod}_{\pl, \leq n}(X)) \otimes \mathbb S
        \ar[r,"(\id \circ_\pl ~ \Delta_{\leq n})\otimes \mathbb{S}"] \ar[d, "\cong"']
        & (X \comp_\pl {\treemod}_{\pl, \leq n}(X) \comp_\pl {\treemod}_{\pl, \leq n}(X)) \otimes \mathbb S
        \ar[d, "\cong"]
        \\
        \overline{\treemod}_{\pl, \leq n+1}(X) \otimes \mathbb S
        \ar[r,"\Delta_{\leq n+1} \otimes \mathbb S"]
        & (\overline{\treemod}_{\pl, \leq n+1}(X) \comp_\pl {\treemod}_{\pl, \leq n+1}(X)) \otimes \mathbb S~,
    \end{tikzcd}
    $$
    where the top horizontal and the bottom horizontal isomorphisms are,respectively: 
    \[
    \overline{\treemod}_{\pl, \leq n+1}(X \otimes \mathbb S) \cong (X \otimes \mathbb S) \comp_\pl \overline{\treemod}_{\pl, \leq n}(X\otimes \mathbb S) \quad \text{and} \quad (X \comp_\pl \overline{\treemod}_{\pl, \leq n}(X)) \otimes \mathbb S \cong \overline{\treemod}_{\pl, \leq n+1}(X) \otimes \mathbb S~,
    \]
    where the first one is given by Corollary \ref{corollary: iso de la récurrence sur les arbres}. One can check that both of these isomorphisms are compatible with the planar cooperad structures, hence the top and the bottom squares commute. The second cell square commutes by induction hypothesis and the third by naturality of $\varphi$. Combined with the fact that the two following square diagrams 
    $$
    \begin{tikzcd}[column sep=2.5pc,row sep=2.5pc]
        \overline{\treemod}_{\pl, \leq n+1}(X \otimes \mathbb S)
        \ar[r,"\cong"] \ar[dd,"\tau_X",swap]
        & (X \otimes \mathbb S) \comp_\pl {\treemod}_{\pl, \leq n}(X\otimes \mathbb S)
        \ar[d,"\id ~\circ_\pl \tau_X"]
        \\
        & (X \otimes \mathbb S) \comp_\pl ({\treemod}_{\pl, \leq n}(X)\otimes \mathbb S)
        \ar[d,"\varphi"]
        \\
        \overline{\treemod}_{\pl, \leq n+1}(X) \otimes \mathbb S
        \ar[r,"\cong"]
        & (X  \comp_\pl {\treemod}_{\pl, \leq n}(X))\otimes \mathbb S~,
    \end{tikzcd}
    $$
    $$
    \begin{tikzcd}[column sep=2.5pc,row sep=3.5pc]
        (X \otimes \mathbb S) \comp_\pl ({\treemod}_{\pl, \leq n}(X)\otimes \mathbb S) \comp_\pl ({\treemod}_{\pl, \leq n}(X)\otimes \mathbb S)
        \ar[r,"\varphi ~\circ_\pl~\id"] \ar[d,"\id ~\circ_\pl~\varphi",swap]
        & ((X \comp_\pl{\treemod}_{\pl, \leq n}(X))\otimes \mathbb S) \comp_\pl ({\treemod}_{\pl, \leq n}(X)\otimes \mathbb S)
        \ar[d,"\varphi"]
        \\
        (X \otimes \mathbb S) \comp_\pl (({\treemod}_{\pl, \leq n}(X) \comp_\pl {\treemod}_{\pl, \leq n}(X))\otimes \mathbb S)
        \ar[r,"\varphi"]
        & (X  \comp_\pl {\treemod}_{\pl, \leq n}(X) \comp_\pl{\treemod}_{\pl, \leq n}(X))\otimes \mathbb S~,
    \end{tikzcd}
    $$
commute, it yields the result at level $n+1$.
\end{proof}

\begin{lemma}\label{lemmaconilpotentcooperadcomonad}
Let $X$ be a dg $\mathbb{N}$-module. The tree intertwining map $\tau$ is compatible with the comonad structure $\delta$ on the reduced tree comonad. That is, the following natural diagrams are commutative

    \[
    \begin{tikzcd}[column sep=2pc,row sep=2.5pc]
        \overline{\treemod}_\pl(X \otimes \mathbb S)
        \ar[rr,"\tau_X"] \ar[d,"\delta_{X \otimes \mathbb{S}}"']
        && \overline{\treemod}_\pl(X) \otimes \mathbb S 
        \ar[d,"\delta_{X} \otimes \mathbb{S}"]
        \\
        \overline{\treemod}_\pl\overline{\treemod}_\pl(X \otimes \mathbb S)
        \ar[r,"\overline{\treemod}_\pl(\tau_X)"]
        & \overline{\treemod}_\pl(\overline{\treemod}_\pl(X) \otimes \mathbb S)
        \ar[r,"\tau_{\overline{\treemod}_\pl(X)}"]
        & \overline{\treemod}_\pl\overline{\treemod}_\pl(X) \otimes \mathbb S~.
    \end{tikzcd}
    \begin{tikzcd}[column sep=2.5pc,row sep=2.5pc]
        \overline{\treemod}_\pl(X \otimes \mathbb S)
        \ar[r,"\tau_X"] \ar[rd,"\eta_{~\overline{\treemod}_\pl(X \otimes \mathbb S)}"']
        & \overline{\treemod}_\pl(X) \otimes \mathbb S 
        \ar[d,"\eta_{~\overline{\treemod}_\pl(X) \otimes \mathbb S}"]
        \\
        &X \otimes \mathbb S~.
    \end{tikzcd}
    \]
    
in the category of dg $\mathbb{N}$-modules.
\end{lemma}

\begin{proof}
The fact that the second diagram is commutative is clear. Let us prove that the first one is commutative.
As in the proof of Lemma \ref{lemmaconilpotentcooperadlemmaone}, we prove the result by an induction on the height of the trees. More precisely, it suffices to prove that the pre-compositions by the inclusion $\overline{\treemod}_{\pl, \leq n}(X \otimes \mathbb S) \hookrightarrow\overline{\treemod}_{\pl}(X \otimes \mathbb S)$ of the arrows in the square diagram are equal. For $n=1$, this result is clear.

\medskip

Let us assume that it is satisfied for up to a natural integer $n \geq 1$. Then, it is satisfied for $n+1$ since the following diagram is commutative

    {\scriptsize{
    $$
    \begin{tikzcd}[column sep=2.5pc,row sep=2.5pc]
        \overline{\treemod}_{\pl, \leq n+1}(X \otimes \mathbb S)
        \ar[rr,"(\tau_X)_{\leq n+1}"] \ar[d,"\Delta_{\leq n+1}",swap]
        && \overline{\treemod}_{\pl, \leq n+1}(X) \otimes \mathbb S 
        \ar[d,"\Delta_{\leq n+1} \otimes \mathbb{S}"]
        \\
        \overline{\treemod}_{\pl, \leq n+1}(X \otimes \mathbb S) \comp_\pl {\treemod}_{\pl, \leq n}(X \otimes \mathbb S)
        \ar[r,"(\tau_X)_{\leq n+1}\circ_\pl(\tau_X)_{\leq n}"]
        \ar[d,"\id ~ \circ_\pl ~ (\delta_{X \otimes \mathbb{S}})_{\leq n}",swap]
        &(\overline{\treemod}_{\pl, \leq n+1}(X) \otimes \mathbb S) \comp_\pl ({\treemod}_{\pl, \leq n}(X) \otimes \mathbb S)
        \ar[r,"\varphi"]\ar[d,"\id \circ_\pl (\delta_{\leq n} \otimes \mathbb{S})"]
        & ({\treemod}_{\pl, \leq n+1}(X)  \comp_\pl \overline{\treemod}_{\pl, \leq n}(X)) \otimes \mathbb S
        \ar[d,"(\id \circ_\pl \delta_{\leq n}) \otimes \mathbb{S}"]
        \\
        \overline{\treemod}_{\pl, \leq n+1}(X \otimes \mathbb S) \comp_\pl {\treemod}_{\pl, \leq n} \overline{\treemod}_{\pl, \leq n}(X \otimes \mathbb S)
        \ar[r,"(\tau_X)_{\leq n+1}\circ_\pl(\tau)_{\leq n}",swap] \ar[d]
        & (\overline{\treemod}_{\pl, \leq n+1}(X) \otimes \mathbb S) \comp_\pl ({\treemod}_{\pl, \leq n} \overline{\treemod}_{\pl, \leq n}(X) \otimes \mathbb S)
        \ar[r,"\varphi"]
        & ({\treemod}_{\pl, \leq n+1}(X)  \comp_\pl {\treemod}_{\pl, \leq n}\overline{\treemod}_{\pl, \leq n}(X)) \otimes \mathbb S
        \ar[d]
        \\
        \overline{\treemod}_{\pl, \leq n+1}\overline{\treemod}_{\pl, \leq n+1}(X \otimes \mathbb S)
        \ar[rr,"\tau_{\overline{\treemod}_{\pl, \leq n+1}(X)}~\overline{\treemod}_\pl((\tau_X)_{\leq n+1})"]
        && (\overline{\treemod}_{\pl, \leq n+1}\overline{\treemod}_{\pl, \leq n+1}(X)) \otimes \mathbb S~,
    \end{tikzcd}
    $$}}
    which is a consequence of Lemma \ref{lemmaconilpotentcooperadlemmaone} and of the induction hypothesis.
\end{proof}

\begin{lemma}\label{lemmatreecomonad}
Let $X$ be a dg $\mathbb{N}$-module. The monad structure $\mu$ of the functor $-\otimes \mathbb{S}$ is compatible with the comonad structure $\delta$ on the reduced tree comonad and the tree intertwining map $\tau$. That is, the following diagram is commutative
    \[
    \begin{tikzcd}[column sep=3.5pc,row sep=2.5pc]
        \overline{\treemod}(X \otimes \mathbb S \otimes \mathbb S)
        \ar[r,"\nu_{X \otimes \mathbb{S}}"] \ar[d,"\overline{\treemod}(\mu_X)"']
        & \overline{\treemod}_\pl(X \otimes \mathbb S )\otimes \mathbb S
        \ar[r,"\delta_{X \otimes \mathbb S}"] \ar[d,"\mu_{\overline{\treemod}_\pl(X)}~(\tau_X \otimes \mathbb S)"]
        & \overline{\treemod}_\pl\overline{\treemod}_\pl(X \otimes \mathbb S) \otimes \mathbb S
        \ar[r,"\overline{\treemod}(\nu^{-1})~ \nu^{-1}_{\overline{\treemod}_\pl}"] \ar[d,"\mu_{\overline{\treemod}_\pl \overline{\treemod}_\pl}~\tau_{\overline{\treemod}_\pl}~\overline{\treemod}_\pl(\tau)"]
        & \overline{\treemod}\overline{\treemod}(X \otimes \mathbb S \otimes \mathbb S)
        \ar[d,"\overline{\treemod}\overline{\treemod}(\mu_X)"]
        \\
        \overline{\treemod}(X \otimes \mathbb S)
        \ar[r,"\nu_{X}"']
        & \overline{\treemod}_\pl(X)\otimes \mathbb S
        \ar[r,"\delta_X \otimes \mathbb{S}"']
        & \overline{\treemod}_\pl\overline{\treemod}_\pl(X)\otimes \mathbb S
        \ar[r,"\overline{\treemod}(\nu^{-1})~ \nu^{-1}_{\overline{\treemod}_\pl}"']
        & \overline{\treemod}\overline{\treemod}(X \otimes \mathbb S).
    \end{tikzcd}
    \]
    in the category of dg $\mathbb{N}$-modules.
\end{lemma}

\begin{proof}
The commutation of the left side square is a consequence of Proposition \ref{corollarytreemodcommutation}. The commutation of the middle square is a consequence of Lemma \ref{lemmaconilpotentcooperadcomonad}. The right square decomposes into the following diagram
    $$
    \begin{tikzcd}[column sep=2.5pc,row sep=2.5pc]
        \overline{\treemod}_\pl\overline{\treemod}_\pl(X \otimes \mathbb S )\otimes \mathbb S
        \ar[r,"\nu^{-1}_{\overline{\treemod}_\pl}"] \ar[d,"\overline{\treemod}_\pl(\tau)",swap]
        & \overline{\treemod}(\overline{\treemod}_\pl(X \otimes \mathbb S) \otimes \mathbb S)
        \ar[d,"\overline{\treemod}(\tau)"] \ar[r,"\overline{\treemod}(\nu^{-1})"]
        & \overline{\treemod}\overline{\treemod}(X \otimes \mathbb S \otimes \mathbb S)
        \ar[dd,"\overline{\treemod}\overline{\treemod}(\mu_X)"]
        \\
        \overline{\treemod}_\pl(\overline{\treemod}_\pl(X) \otimes \mathbb S )\otimes \mathbb S
        \ar[r,"\nu^{-1}_{\overline{\treemod}_\pl}"] \ar[d,"\mu_{\overline{\treemod}_\pl \overline{\treemod}_\pl}~\tau_{\overline{\treemod}_\pl}",swap]
        & \overline{\treemod}(\overline{\treemod}_\pl(X) \otimes \mathbb S \otimes \mathbb S )
        \ar[d,"\overline{\treemod}(\mu_{\overline{\treemod}_\pl})"]
        \\
        \overline{\treemod}_\pl\overline{\treemod}_\pl(X)\otimes \mathbb S
        \ar[r,"\nu^{-1}_{\overline{\treemod}_\pl}"]
        & \overline{\treemod}(\overline{\treemod}_\pl(X) \otimes \mathbb S )
        \ar[r,"\overline{\treemod}(\nu^{-1})"]
        & \overline{\treemod}\overline{\treemod}(X \otimes \mathbb S)~,
    \end{tikzcd}
    $$
    which is commutative. Indeed, its top left cell is clearly commutative, its bottom left cell is the left square in the lemma applied to $\overline{\treemod}_\pl(X)$, and its right cell is the image by $\overline{\treemod}(-)$ of the same left square in the lemma.
\end{proof}

\begin{theorem}
The reduced tree endofunctor $\overline{\treemod}$ on the category of dg $\mathbb S$-module has the canonical structure of a comonad. For a dg $\mathbb{S}$-module $M$, 

\begin{enumerate}
\item its counit $\epsilon: \overline{\treemod}(M) \longrightarrow M$ is the canonical projection into the corollas,

\medskip

\item its coproduct $\delta_M: \overline{\treemod}(M) \longrightarrow \overline{\treemod}\overline{\treemod}(M)$ is the unique arrow which makes the following diagram 
 \[
    \begin{tikzcd}[column sep=3.25pc,row sep=3.5pc]
        \overline{\treemod}(M)
        \ar[rrr,"\delta_M"]
        &&&
        \overline{\treemod}\overline{\treemod}(M)
        \\
        \overline{\treemod}(\mathrm{U}(M) \otimes \mathbb{S})
        \ar[u,"\overline{\treemod}(\pi_M)"] \ar[r,"\nu_{\mathrm{U}(M)}"]
        & \overline{\treemod}_\pl(\mathrm{U}(M)) \otimes \mathbb S
        \ar[r,"\delta_{\mathrm{U}(M)} \otimes \mathbb{S}"]
        &\overline{\treemod}_\pl\overline{\treemod}_\pl(\mathrm{U}(M)) \otimes \mathbb S
        \ar[r,"\overline{\treemod}(\nu^{-1})~ \nu^{-1}_{\overline{\treemod}_\pl}"]
        &\overline{\treemod}\overline{\treemod}(\mathrm{U}(M) \otimes \mathbb{S}).
        \ar[u,"\overline{\treemod}\overline{\treemod}(\pi_M)"']
    \end{tikzcd}
    \]
commute, where $\pi_M: \mathrm{U}(M) \otimes \mathbb{S} \longrightarrow M$ is the map given by the dg $\mathbb{S}$-module structure.
\end{enumerate}
\end{theorem}
    
\begin{proof}
The natural map 
\[
\begin{tikzcd}[column sep=3.5pc,row sep=2.5pc]
        \overline{\treemod}(\mathrm{U}(M) \otimes \mathbb{S})
        \ar[r,"\nu_{\mathrm{U}(M)}"]
        & \overline{\treemod}_\pl(\mathrm{U}(M)) \otimes \mathbb S
        \ar[r,"\delta_{\mathrm{U}(M)} \otimes \mathbb{S}"]
        &\overline{\treemod}_\pl\overline{\treemod}_\pl(\mathrm{U}(M)) \otimes \mathbb S
        \ar[r,"\overline{\treemod}(\nu^{-1})~ \nu^{-1}_{\overline{\treemod}_\pl}"]
        &\overline{\treemod}\overline{\treemod}(\mathrm{U}(M) \otimes \mathbb{S})
\end{tikzcd}
\]
induces a morphism of diagrams
\[
\left(\begin{tikzcd}
    \overline{\treemod}(\mathrm{U}(M) \otimes \mathbb S \otimes \mathbb S)
    \ar[d,"\overline{\treemod}(\pi_M)", shift left=1.1ex] \ar[d,"\overline{\treemod}(\mu_{\mathrm{U}(M)})",swap,shift right=1.1ex]
    \\
    \overline{\treemod}(\mathrm{U}(M) \otimes \mathbb S)
\end{tikzcd}\right)
\longrightarrow
\left(\begin{tikzcd}
    \overline{\treemod}\overline{\treemod}(\mathrm{U}(M) \otimes \mathbb S \otimes \mathbb S)
    \ar[d,"\overline{\treemod}\overline{\treemod}(\pi_M)", shift left=1.1ex] \ar[d,"\overline{\treemod}\overline{\treemod}(\mu_{\mathrm{U}(M)})", swap, shift right=1.1ex]
    \\
    \overline{\treemod}\overline{\treemod}(\mathrm{U}(M) \otimes \mathbb S)
\end{tikzcd}\right)
\]
as a consequence of Lemma \ref{lemmatreecomonad}. Taking the colimits of these diagrams, one obtains a natural coproduct map
\[
\delta_M: \overline{\treemod}(M) \longrightarrow \overline{\treemod}\overline{\treemod}(M)~.
\]
It remains to show that this coproduct map and the counit endow the endofunctor $\overline{\treemod}(-)$ with a comonad structure.

\medskip

The counitality is straightforward to check. Let us prove that the coproduct map is coassociative. Since $\overline{\treemod}(-)$ preserves epimorphisms and since the map $\mathrm{U}(M) \otimes \mathbb S \longrightarrow M$ is an epimorphism, it suffices to prove that 
the following square diagram
\[
\begin{tikzcd}[column sep=2.5pc,row sep=2.5pc]
    \overline{\treemod}(\mathrm{U}(M) \otimes \mathbb S)
    \ar[r, "\delta_{\mathrm{U}(M) \otimes \mathbb S}"] \ar[d, "\delta_{\mathrm{U}(M) \otimes \mathbb S}"']
    & \overline{\treemod} \overline{\treemod}(\mathrm{U}(M) \otimes \mathbb S)
    \ar[d, "\delta_{\overline{\treemod}(\mathrm{U}(M) \otimes \mathbb S)}"]
    \\
    \overline{\treemod} \overline{\treemod}(\mathrm{U}(M) \otimes \mathbb S)
    \ar[r, "\overline{\treemod}(\delta_{\mathrm{U}(M) \otimes \mathbb S})"']
    & \overline{\treemod}\overline{\treemod}\overline{\treemod}(\mathrm{U}(M) \otimes \mathbb S)
\end{tikzcd}
\]
is commutative. This follows from the coassociativity of coproduct map $\delta_\pl: \overline{\treemod}_\pl \longrightarrow \overline{\treemod}_\pl\overline{\treemod}_\pl$ and from the fact that the natural isomorphism
\[
\nu_{\mathrm{U}(M)}^{-1}: \overline{\treemod}(\mathrm{U}(M))_\pl \otimes \mathbb S \longrightarrow \overline{\treemod}(\mathrm{U}(M) \otimes \mathbb S)
\]
is suitably compatible with the coproducts.
\end{proof}

\begin{remark}
One can notice that the natural isomorphism 
\[
\nu: \overline{\treemod}(- \otimes \mathbb S) \longrightarrow \overline{\treemod}(-)_\pl \otimes \mathbb S
\]
between functors from the category of dg $\mathbb{N}$-modules to the category of dg $\mathbb S$-modules is a comonad transformation above the functor $- \otimes \mathbb S$. More precisely, this amounts to the following diagrams commuting
\[
    \begin{tikzcd}
        \overline{\treemod}(-)_\pl \otimes \mathbb S
        \ar[r] \ar[d]
        & \overline{\treemod}(- \otimes \mathbb S)
        \ar[d]
        \\
        \overline{\treemod}\overline{\treemod}(-)_\pl \otimes \mathbb S
        \ar[r]
        & \overline{\treemod}\overline{\treemod}(- \otimes \mathbb S)
    \end{tikzcd}
    \quad 
    \begin{tikzcd}
        \overline{\treemod}(-)_\pl \otimes \mathbb S
        \ar[r] \ar[rd]
        & \overline{\treemod}(- \otimes \mathbb S)
        \ar[d]
        \\
        & - \otimes \mathbb S~.
    \end{tikzcd}
\]
This induces by (co)doctrinal adjunction the following corollary.
\end{remark}

\begin{corollary}\label{corollaryadjunctionconilp}
The adjunction 
\[
\begin{tikzcd}[column sep=3.5pc,row sep=0.5pc]
          \mathsf{dg}~\mathbb{N}\text{-}\mathsf{mod} \arrow[r, shift left=1.1ex, "- \otimes \mathbb S"{name=A}]
          &\mathsf{dg}~\mathbb{S}\text{-}\mathsf{mod}  \arrow[l, shift left=.75ex, "\mathrm{U}"{name=B}]
             \arrow[phantom, from=A, to=B, , "\dashv" rotate=90]
\end{tikzcd}
\]
lifts canonically to an adjunction 
\[
\begin{tikzcd}[column sep=3.5pc,row sep=0.5pc]
          \mathsf{dg}~\overline{\treemod}_\pl\text{-}\mathsf{mod} \arrow[r, shift left=1.1ex, "- \otimes \mathbb S"{name=A}]
          &\mathsf{dg}~\overline{\treemod}\text{-}\mathsf{cog} ~. \arrow[l, shift left=.75ex, "U_{\mathbb S}"{name=B}]
             \arrow[phantom, from=A, to=B, , "\dashv" rotate=90]
\end{tikzcd}
\]
\end{corollary}

\begin{proof}
Let $(X,\delta_X)$ be a dg $\overline{\treemod}_\pl$-coalgebra. The dg $\mathbb{S}$-module inherits a dg $\overline{\treemod}$-coalgebra structure given by
\[
\begin{tikzcd}[column sep=2.5pc,row sep=0.5pc]
X \otimes \mathbb{S} \arrow[r,"\delta_X \otimes \mathbb{S}"]
&\overline{\treemod}_\pl (X)  \otimes \mathbb S \arrow[r,"\nu_X"]
&\overline{\treemod} (X \otimes \mathbb S)~.
\end{tikzcd}
\]
Let $(M,\delta_M)$ be a dg $\overline{\treemod}$-coalgebra. The dg $\mathbb{N}$-module $\mathrm{U}(M)$ inherits a dg $\overline{\treemod}_\pl$-coalgebra given by 
\[
\begin{tikzcd}[column sep=2.5pc,row sep=0.5pc]
\mathrm{U}(M) \arrow[r,"\mathrm{U}(\delta_M)"]
&\mathrm{U}(\overline{\treemod}(M)) \arrow[r,"\overline{\treemod}(\eta_M)"]
&\mathrm{U}(\overline{\treemod} (\mathrm{U}(M)  \otimes \mathbb S)) \arrow[r,"\nu_{\mathrm{U}(M)}"]
&\overline{\treemod}_\pl (\mathrm{U}(M)) \otimes \mathbb S) \arrow[r,"\overline{\treemod}_\pl (\pi_M)"]
&\overline{\treemod}_\pl (\mathrm{U}(M))~.
\end{tikzcd}
\]
One can check that these two natural constructions produce adjoint functors.
\end{proof}

\begin{corollary}
For every dg $\overline{\treemod}$-coalgebra $(W,\delta_W)$, the following square
    \[
    \begin{tikzcd}[column sep=2.5pc,row sep=2.5pc]
        W
        \ar[r,"\delta_W"] \ar[d,"\delta_W",swap]
        & \overline{\treemod} W
        \ar[r, two heads]
        & \overline{\treemod}_{\leq 2} W
        \ar[r, "\cong"]
        & W \comp (\operad I \oplus W)
        \ar[d,rightarrowtail]
        \\
        \overline{\treemod} W
        &&&  W \comp {\treemod} W.
        \ar[lll, "\cong",swap]
    \end{tikzcd}
    \]
is a commutative diagram.
\end{corollary}

\begin{proof}
This follows from the same arguments, \textit{mutatis mutandis}, as those used in the planar context for the proof of Lemma \ref{lemma: T-cog implique comonoide pour le circ}.
\end{proof}

\subsection{Conilpotent cooperads}
Let $(W,\delta_W)$ be a dg $\overline{\treemod}$-coalgebra. We consider the map
\[
\Delta_W: \operad I \oplus W \longrightarrow (\operad I \oplus W)\comp (\operad I \oplus W)
\]
defined as the sum of the maps

\medskip

\begin{tikzcd}[column sep=1.5pc,row sep=0.5pc]
    \operad I \oplus W \arrow[r,twoheadrightarrow]
    &W \arrow[r,"\delta_W"]
    &\overline{\treemod} W \arrow[r,twoheadrightarrow]
    &\overline{\treemod}_{\leq 2} W \cong W \comp (\operad I \oplus W) \arrow[r,rightarrowtail]
    &(\operad I \oplus W)\comp (\operad I \oplus W)~,
\end{tikzcd}

\begin{tikzcd}[column sep=2pc,row sep=0.5pc]
    \operad I \oplus W \arrow[r,twoheadrightarrow]
    &W \cong \operad I \comp W \arrow[r,rightarrowtail]
    &(\operad I \oplus W)\comp (\operad I \oplus W)~,
\end{tikzcd}

\begin{tikzcd}[column sep=2pc,row sep=0.5pc]
    \operad I \oplus W \arrow[r,twoheadrightarrow]
    &\operad I \arrow[r]
    &\operad I \comp \operad I~.
\end{tikzcd}

\medskip

Together with the counit $\epsilon_W: \operad I \oplus W \twoheadrightarrow \operad I$ and the coaugmentation $\mu_W: \operad I \rightarrowtail \operad I \oplus W$, they form a coaugmented cooperad structure on the dg $\mathbb S$-module $\operad I \oplus W$.

\medskip

This defines a functor 
\[
\mathrm{Conil}: \mathsf{dg}~\overline{\treemod}\text{-}\mathsf{cog} \longrightarrow (\dgcooperads)_{\operad I/}
\]
from dg $\overline{\treemod}$-coalgebras to coaugmented dg cooperads.

\begin{proposition}\label{propositionplanarccooperadff}
    The functor $\mathrm{Conil}$ from dg $\overline{\treemod}$-coalgebras to coaugmented dg cooperads
    is fully faithful.
\end{proposition}

\begin{proof}
This follows from the same arguments, \textit{mutatis mutandis}, as those used in the planar context. See the proof of Proposition \ref{propositionplanarccooperadff}.
\end{proof}

\begin{definition}[Conilpotent dg cooperad]
Let $\C$ be a coaugmented dg cooperad. It is \textit{conilpotent} if it belongs to the essential image of the functor $\mathrm{Conil}$ from dg $\overline{\treemod}$-coalgebras to dg cooperads. We denote $\dgcooperads^{\categ{conil}}$ the full sub-category of coaugmented dg cooperads spanned by conilpotent ones.
\end{definition}

\begin{corollary}
The adjunction 
\[
\begin{tikzcd}[column sep=5pc,row sep=3pc]
          \dgcooperads_\pl  \arrow[r, shift left=1.1ex, "-\otimes \mathbb S"{name=F}] &\dgcooperads, \arrow[l, shift left=.75ex, "\mathrm{U}_{\mathbb S}"{name=U}]
            \arrow[phantom, from=F, to=U, , "\dashv" rotate=90]
\end{tikzcd}
\]
restricts to an adjunction 
\[
\begin{tikzcd}[column sep=5pc,row sep=3pc]
          \dgcooperads_\pl^{\categ{conil}}  \arrow[r, shift left=1.1ex, "-\otimes \mathbb S"{name=F}] &\dgcooperads^{\categ{conil}}, \arrow[l, shift left=.75ex, "\mathrm{U}_{\mathbb S}"{name=U}]
            \arrow[phantom, from=F, to=U, , "\dashv" rotate=90]
\end{tikzcd}
\]
between conilpotent planar dg cooperads and conilpotent dg cooperads.
\end{corollary}

\begin{proof}
This follows from the definitions.
\end{proof}

\begin{corollary}
Let $n$ be a natural integer. 

\medskip

\begin{enumerate}
\item The comonad structure on the reduced tree endofunctor $\overline{\treemod}$ restricts to $\overline{\treemod}_{\leq n}$ and $\overline{\treemod}^{(\leq n)}$. 

\medskip

\item The adjunction 
\[
\begin{tikzcd}[column sep=3.5pc,row sep=0.5pc]
          \mathsf{dg}~\overline{\treemod}_\pl\text{-}\mathsf{mod} \arrow[r, shift left=1.1ex, "- \otimes \mathbb S"{name=A}]
          &\mathsf{dg}~\overline{\treemod}\text{-}\mathsf{cog} ~, \arrow[l, shift left=.75ex, "U_{\mathbb S}"{name=B}]
             \arrow[phantom, from=A, to=B, , "\dashv" rotate=90]
\end{tikzcd}
\]
restricts to adjunctions 
\[
\begin{tikzcd}[column sep=3.5pc,row sep=0.5pc]
          \mathsf{dg}~\overline{\treemod}_\pl^{(\leq n)}\text{-}\mathsf{mod} \arrow[r, shift left=1.1ex, "- \otimes \mathbb S"{name=A}]
          &\mathsf{dg}~\overline{\treemod}^{(\leq n)} \text{-}\mathsf{cog} ~, \arrow[l, shift left=.75ex, "U_{\mathbb S}"{name=B}]
             \arrow[phantom, from=A, to=B, , "\dashv" rotate=90]
\end{tikzcd}
\]
\[
\begin{tikzcd}[column sep=3.5pc,row sep=0.5pc]
          \mathsf{dg}~\overline{\treemod}_{\pl, \leq n}\text{-}\mathsf{mod} \arrow[r, shift left=1.1ex, "- \otimes \mathbb S"{name=A}]
          &\mathsf{dg}~\overline{\treemod}_{\leq n}\text{-}\mathsf{cog} ~. \arrow[l, shift left=.75ex, "U_{\mathbb S}"{name=B}]
             \arrow[phantom, from=A, to=B, , "\dashv" rotate=90]
\end{tikzcd}
\]
\end{enumerate}
\end{corollary}

\begin{proof}
This follows from the definition of the coproduct of the comonad $\overline{\treemod}$. See Corollary \ref{corollarydecompositionplanarccooperad}.
\end{proof}

\begin{proposition}\label{proppresentabilityofplanarconilp}
    Let $n$ be a natural integer. The forgetful functor 
    \[
    \mathsf{dg}~\overline{\treemod}^{(\leq n)}\text{-}\mathsf{cog} \longrightarrow \mathsf{dg}~\overline{\treemod}\text{-}\mathsf{cog}~.
    \]
    is fully faithful. Furthermore, both categories are presentable, and therefore it admits a right adjoint denoted by $\overline{\mathrm{F}}_n$.
\end{proposition}

\begin{proof}
This follows from the same arguments as in the planar case, see Proposition \ref{proppresentabilityofplanarconilp}.
\end{proof}

We will denote by $\mathrm{F}^{\mathrm{rad}}_n$ the following endofunctor 
\[
\begin{tikzcd}
\dgcooperads^{\categ{conil}} \arrow[r,"\cong"]
&\mathsf{dg}~\overline{\treemod}\text{-}\mathsf{cog} \arrow[r,"\overline{\mathrm{F}}_n"]
&\mathsf{dg}~\overline{\treemod}^{(\leq n)}\text{-}\mathsf{cog}  \arrow[r]
&\mathsf{dg}~\overline{\treemod}\text{-}\mathsf{cog}  \arrow[r,"\cong"]
&\dgcooperads^{\categ{conil}}
\end{tikzcd}
\]

of the category of conilpotent dg cooperads.

\begin{definition}[Coradical filtration]
Let $\C$ be a conilpotent dg cooperad. Its $n$-\textit{coradical filtration} is given by the conilpotent dg cooperad $\mathrm{F}^{\mathrm{rad}}_n \C$. It induces a ladder diagram 
\[
\mathrm{F}^{\mathrm{rad}}_0 \C \rightarrowtail \mathrm{F}^{\mathrm{rad}}_1 \C \rightarrowtail \cdots \mathrm{F}^{\mathrm{rad}}_n \C \rightarrowtail \cdots~,
\]
indexed by $\mathbb{N}$, where all the arrows are monomorphisms.
\end{definition}

For any  conilpotent dg cooperad $\C$, there is a canonical isomorphism between $\C$ and the colimit of the following ladder diagram
\[
\mathrm{F}^{\mathrm{rad}}_0 \C \rightarrowtail \mathrm{F}^{\mathrm{rad}}_1 \C \rightarrowtail \cdots \mathrm{F}^{\mathrm{rad}}_n \C \rightarrowtail \cdots
\]
in the category of conilpotent dg cooperads.

\begin{remark}
If the characteristic of the base field $\kk$ is zero, both the tree module $\treemod (-)$ and the composition product $\circ$ preserve {\bf{finite}} cosifted limits. Therefore the forgetful functors from (conilpotent) dg cooperads to dg $\mathbb S$-modules preserve these limits. Moreover, as in the planar case, given a dg cooperad $\C$, the conilpotent dg cooperad $\overline{\mathrm{F}}_n \C$ fits in the following pullback diagram
\[
        \begin{tikzcd}[column sep=2.5pc,row sep=2.5pc]
           \mathrm{F}^{\mathrm{rad}}_n \C \arrow[dr, phantom, "\lrcorner", very near start]
            \ar[r] \ar[d]
            & \overline{\treemod}^{(\leq n)} \C
            \ar[d,rightarrowtail]
            \\
            \C
            \ar[r,"\delta_\C"]
            & \overline{\treemod} \C~,
        \end{tikzcd}
\]
in the category of dg $\mathbb{S}$-modules, where $\delta_\C$ denotes the dg $\overline{\treemod}$-coalgebra structure of $\C$.
\end{remark}

\appendix
\section{Trees}
\label{appendixtrees}
The goal of this appendix is to give so recollections on trees. 

\begin{definition}[Planar tree]
    A \textit{planar tree} $t$ is a tuple $(\mathrm{edges}(t), \mathrm{leaves}(t), <, \tilde <)$ where
    \begin{itemize}
        \item $(\mathrm{edges}(t), <)$ is a finite non-empty poset called the poset of edges of $t$. It has a least element called the root. Moreover, for every edge $e$, the subposet spanned by elements that are lower that $e$ $\{e' \in \mathrm{edges}(t)|\ e' \leq e\}$ is a linear poset;
        \item $\mathrm{leaves}(t)$ is a subset of the set of maximal elements of the set of edges;
        \item $\tilde <$ is an additional order on the set of edges, called the planar order, that is linear and so that 
        \begin{itemize}
            \item for every two edges $e,e'$, $e ~ \tilde < ~ e'$ whenever $e < e'$;
            \item for every three edges $e,e',e''$ so that $e < e'$ and $e'' \tilde < e'$, then $e$ and $e''$ are comparable for the order $<$ in the sense that either $e'' < e$ or $e'' = e$ or  $e'' > e$.
        \end{itemize}
    \end{itemize}
\end{definition}

\begin{definition}[Node]
Let $t$ be a planar tree. A \textit{node} (or \textit{vertex}) of $t$ is an edge $e$ that is not a leaf. The inputs of this nodes are the edges $e'$ that are strictly greater than $e$ so that there are no edge between $e$ and $e'$. Sometimes, the edge that defines the node is called the output edge of this node. Moreover, the number of inputs of a node is called its arity.
\end{definition}

A drawing is often simpler than a definition. Let us consider the following one.
 
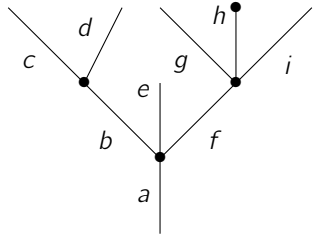
\begin{center}
\begin{tikzpicture}
    \draw (2,0) -- (2,2) ;
    \draw (2,1) -- (1,2) ;
    \draw (1,2) -- (1.5,3) ;
    \draw (1,2) -- (0,3) ;
    \draw (2,1) -- (3,2) ;
    \draw (3,2) -- (2,3) ;
    \draw (3,2) -- (3,3) ;
    \draw (3,2) -- (4,3) ;
    \draw (2,1) node {$\bullet$} ;
    \draw (1,2) node {$\bullet$} ;
    \draw (3,2) node {$\bullet$} ;
    \draw (3,3) node {$\bullet$} ;
    \draw (2,0.5) node[left]{$a$} ;
    \draw (1.5,1.5) node[below left]{$b$} ;
    \draw (0.5,2.5) node[below left]{$c$} ;
    \draw (1.25,2.5) node[above left]{$d$} ;
    \draw (2,1.9) node[left]{$e$} ;
    \draw (2.5,1.5) node[below right]{$f$} ;
    \draw (2.5,2.5) node[below left]{$g$} ;
    \draw (3,2.9) node[left]{$h$} ;
    \draw (3.5,2.5) node[below right]{$i$} ;
\end{tikzpicture}
\end{center}
It represents a planar tree whose set of edges with planar order is 
$$
a ~ \tilde < ~ b ~ \tilde < ~ c ~ \tilde < ~ d ~ \tilde < ~ e ~ \tilde < ~ f ~ \tilde < ~ g ~ \tilde < ~ h ~ \tilde < ~ i~.
$$
For the tree order is made up of the relations
$$
a < b,e,f ; \quad b < c,d ; \quad f< g,h,i.
$$
The leaves are $c,d,e,g,i$ and the nodes are $a,b,f,h$ with arities respectively $3, 2,3, 0$.

\begin{definition}[Tree isomorphisms]
Let $t,t'$ be two trees. An i\textit{isomorphism of trees} from $t$ to $t'$ is an isomorphism of posets from $(\mathrm{edges}(t), <)$ to $(\mathrm{edges}(t'), <)$ that restricts to a bijection between leaves; thus it restricts to a bijection between nodes.

\medskip

An \textit{isomorphism of planar trees} from $t$ to $t'$ is an isomorphism of trees that preserves also the planar order.
\end{definition}

\begin{definition}[Tree groupoids]
    The groupoid of planar trees $\categ{Trees}_\pl$
    is the groupoid whose objects are planar trees and whose
    morphisms are isomorphisms of planar trees.
    
    \medskip
    
    The groupoid of trees $\categ{Trees}$
    is the groupoid whose objects are planar trees and whose
    morphisms are isomorphisms of trees.
\end{definition}

\begin{definition}[Trivial tree]
    A tree $t$ is \textit{trivial} if $\mathrm{edges}(t) = \mathrm{leaves}(t) \cong \ast$.
\end{definition}

\begin{definition}[Subtree]
    Let $t$ be a planar tree. A \textit{subtree} $t'$ of $t$ is the data of a tuple
    $(\mathrm{edges}(t'), \mathrm{leaves}(t'), <, \tilde <)$ where
    \begin{itemize}
        \item $\mathrm{edges}(t')$ is a subset of $\mathrm{edges}(t)$ and $<, \tilde <$ are the orders induced from those of $t$;
        \item $\mathrm{leaves}(t)$ is a subset of the set of maximal elements of the poset of edges of $(t', <)$;
        \item this two subsets satisfy the following conditions
        \begin{itemize}
            \item the poset $(\mathrm{edges}(t'), <)$ has a least element (the root of $t'$);
            \item for every three edges $e < e' < e''$ of $t$, $e'$ belongs to $t'$ whenever $e$ and $e''$ do;
            \item for every node $n$ of $t$ that belongs to $t'$, either the underlying edge of $n$ is a leaf of $t'$ or $t'$ has in its edges all the inputs of $n$;
            \item a leaf of $t$ that belongs to the edges of $t'$ is actually a leaf of $t'$ that is, $\mathrm{leaves}(t) \cap \mathrm{edges}(t') \subset \mathrm{leaves}(t)$.
        \end{itemize}
        In particular, $t'$ is a planar tree.
    \end{itemize}
\end{definition}

\begin{definition}[Partition of trees]
    Let $t$ be a non-trivial planar tree. A \textit{partition of} $t$ into non trivial subtrees is the data of a natural integer $m \geq 1$ together with non trivial subtrees $t_1, \ldots, t_m$ so that
    \begin{itemize}
        \item for every node $n$ of $t$, there exists a unique $i$ so that $n$ is a node of $t_i$;
        \item for every leaf $l$ of $t$, there exists a unique $i$ so that $l$ is an edge (and thus a leaf) of $t_i$.
    \end{itemize}
\end{definition}

\begin{proposition}
    Let $t$ be a non-trivial planar tree and let $t_1, \ldots, t_m$ be a partition of $t$. Then
    \begin{itemize}
        \item every node edge $e$ of $t$ belongs to at least one and at  most two of these subtrees;
        \item the root of $t$ belongs to only one subtree;
        \item the root of a subtree is a node of this subtree and thus a node of $t$ 
        \item an node edge of $t$ can be the root of only one subtree. 
    \end{itemize}
\end{proposition}

\begin{proof}
This is a straightforward consequence of the definition of a partition into non-trivial subtrees.
\end{proof}

\begin{definition}[Quotient tree]
    Let $t$ be a non-trivial planar tree and let $t_1, \ldots, t_m$ be a partition of $t$. The \textit{quotient tree}  $t/(t_1, \ldots, t_m)$ of $t$ by this partition is the tree whose nodes are the roots of the subtrees, the leaves are those of $t$ and whose orders $<, \tilde <$ are those of $t$.  
\end{definition}

\begin{definition}[Grafting tree]
    Let $t$ be a non-trivial planar subtree with $k$ nodes $t_1, \ldots, t_k$ be non-trivial trees so that the number of leaves of $t_i$ is equal to the number of inputs of the $ith$ node. This gives a canonical isomorphism
    $$
    \mathrm{edges}(t) \backslash \{\mathrm{root}(t)\} \cong  \coprod_{i=1}^k \mathrm{leaves}(t_i).
    $$
    Then we define the \textit{grafting tree}
    $\mathrm{graft}_t(t_1, \ldots, t_k)$ of the trees $t_1, \ldots, t_k$
    along $t$ as the planar tree so that
    \begin{itemize}
        \item the set of node edges is $\coprod_{i=1}^k \mathrm{nodes}(t_i)$;
        \item the set of leaves is the subset of $\coprod_{i=1}^k \mathrm{leaves}(t_i)$ made up of elements whose image in $\mathrm{edges}(t)$ is a leaf of $t$;
        \item for every two edges $e,e'$ of $\mathrm{graft}_t(t_1, \ldots, t_k)$ that proceeds from the two trees $t_i,t_j$ attached to the nodes $n_i, n_j$ of $t$, $e<e'$ 
        if and only if either
        \begin{itemize}
            \item $n_i < n_j$ and $e < l$ where $l$ is the leaf of $t_i$ that corresponds to the unique input $x$ of the node $n_i$ of $t$ so that $x \leq n_j$;
            \item $i=j$ and $e < e'$ within $t_i$;
        \end{itemize}
        \item for every two edges $e,e'$ of $\mathrm{graft}_t(t_1, \ldots, t_k)$ that proceeds from the two trees $t_i,t_j$ attached to the nodes $n_i, n_j$ of $t$, $e<e'$ 
        if and only if either
        \begin{itemize}
            \item $n_i ~ \tilde < ~ n_j$ and $n_i$ and $n_j$ are not comparable for $<$;
            \item $n_i < n_j$ and $e ~ \tilde < ~ l$ where $l$ is the leaf of $t_i$ that corresponds to the unique input $x$ of the node $n_i$ of $t$ so that $x \leq n_j$;
            \item $i=j$ and $e~ \tilde < ~e'$ within $t_i$.
        \end{itemize}
    \end{itemize}
\end{definition}

\begin{proposition}
Let $t$ be a non-trivial planar tree and let $t_1, \ldots, t_m$ be a partition of $t$. One has a canonical isomorphism of planar trees
    $$
    t \cong \mathrm{graft}_{t'}(t_1, \ldots, t_k)
    $$
    where $t'= t/(t_1, \ldots, t_k)$.
\end{proposition}

\begin{proof}
This follows from a straightforward check.
\end{proof}

\bibliographystyle{alpha}
\bibliography{bibax}


\end{document}